\documentclass[11pt,reqno,letterpaper]{amsart}
\usepackage[all]{xy}
\usepackage{hyperref}
\usepackage{amsfonts,mathrsfs,bbm,latexsym,rawfonts,amsmath,amssymb,amsthm,epsfig}
\usepackage{setspace, color}
\usepackage{cases}
\usepackage{todonotes}

\newtheorem{thm}{Theorem}[section]
\newtheorem{cor}[thm]{Corollary}
\newtheorem{rem}[thm]{Remark}
\newtheorem{lem}[thm]{Lemma}

\numberwithin{equation}{section}

\newcommand{\al}{\alpha}

\newcommand{\Ld}{\Lambda}
\newcommand{\de}{\delta}
\newcommand{\De}{\Delta}
\newcommand{\ep}{\varepsilon}

\newcommand{\si}{\sigma}
\newcommand{\om}{\omega}
\newcommand{\Om}{\Omega}
\newcommand{\ga}{\gamma}
\newcommand{\Ga}{\Gamma}
\newcommand{\ka}{\kappa}
\renewcommand{\th}{\theta}

\newcommand{\F}{\mathcal{F}}
\newcommand{\G}{\mathcal{G}}
\renewcommand{\P}{\mathcal{P}}

\renewcommand{\S}{\mathscr{S}}
\newcommand{\g}{\mathfrak{g}}

\newcommand{\U}{\mathbb{S}}

\DeclareMathOperator{\tr}{tr}

\newcommand{\Real}{\mathbb{R}}

\newcommand{\norm}[1]{\Vert#1\Vert}

\def\<{\left\langle} \def\>{\right\rangle}
\def\({\left(} \def\){\right)}
\newcommand{\n}{\nabla}
\newcommand{\p}{\partial}

\newcommand{\A}{\mathscr{A}}
\renewcommand{\S}{\mathscr{S}}

\begin{document}
\title[]{Isolated singularities of 3-dimensional Yang-Mills-Higgs fields}
\author[B. Chen]{Bo Chen}
\address{School of Mathematics, South China University of Technology, Guangzhou, 510640, People's Republic of China}
\email{cbmath@scut.edu.cn}

\author{Chong Song}
\address{School of Mathematical Sciences, Xiamen University, Xiamen, 361005, People's Republic of China}
\email{songchong@xmu.edu.cn}

\begin{abstract}
In this paper, we derive decay estimates near isolated singularities of 3-dimensional (3d) Yang-Mills-Higgs fields defined on a fiber bundle, where the fiber space is a compact Riemannian manifold and the structure group is a connected compact Lie group. As an application, we obtain removable singularity theorems for 3d Yang-Mills-Higgs fields under different types of energy conditions, which generalizes classical removable singularity theorems for 3d Yang-Mills fields~\cite{S84,SS84} and 3d harmonic maps~\cite{L85}.
\end{abstract}
\maketitle
\section{Introduction}
\subsection{The Yang-Mills-Higgs functional and backgrounds}
Let $G$ be a compact connected Lie group and $\g$ be its Lie algebra. Let $(N,h)$ be a Riemannian manifold, which supports a left action of $G$ preserving the metric $h$. Let $(\P,\pi)$ be a $G$-principal bundle over a Riemannian manifold $(M, g)$ and $(\F=\P\times_{G}N, \pi_{\F})$ be the associated Higgs bundle. Let $V: N \to \Real$ be a $G$-invariant function, which gives rise to a so-called Higgs potential, i.e. a gauge invariant function defined on the bundle $\F$.

For a smooth connection $A$ on $\P$ and a smooth section $\Phi$ of $\F$,  the Yang-Mills-Higgs (YMH) functional is defined by
\begin{equation}\label{YMHF0}
	\mathscr{E}(A,\Phi)=\int_{M}|F_A|^2d\mu_g+\int_{M}|\n_A \Phi|^2d\mu_g+\int_{M}V(\Phi)d\mu_g,
\end{equation}
where $F_A$ and $\n_A$ are the curvature and  the covariant exterior derivative induced by the connection $A$ respectively, and $d\mu_g$ is the volume element of $M$.
A Yang-Mills-Higgs field $(A,\Phi)$ is a critical point of the YMH functional, which solves the following elliptic system
\begin{equation}\label{eq-ymh0}
\begin{cases}
D^*_A F_A=-\Phi^* \n_A\Phi, \\[1ex]
\n_A^* \n_A\Phi=-\n V(\Phi),
\end{cases}
\end{equation}
where $D^*_A$ is the dual operator of the extrinsic derivative $D_A$ and the term $\Phi^*\n_A\Phi$ lies in the dual space of $\Om^1(\P\times_{ad}\g)$, namely, for all $B\in \Om^1(\P\times_{ad}\g)$, we have
\[ \<\Phi^*\n_A\Phi, B\>=\<\n_A\Phi, B\Phi\>.\]
Here $B\Phi$ is induced by the infinitesimal action of $\g$-coefficients of $B$ on $\Phi$.

The YMH functional, originating from electromagnetic phenomena research, plays a fundamental role in the classical gauge theory and in particle physics. YMH fields illuminate diverse gauge field models, including the Ginzburg-Landau model in superconductivity theory and the Chern-Simons model in quantum physics, depending on bundles and Higgs potentials. See for example \cite{Taub,M97,EH,EH1}.

Mathematically, YMH fields provide a unified framework, generalizing both classical Yang-Mills (YM) fields and harmonic maps. Specifically, YMH fields reduce to pure YM fields when the fiber space collapses to a point and to harmonic maps from the base manifold $M$ to the fiber space $N$ when the structure group $G$ is trivial. In the context of symplectic fibrations on a surface, minimizing YMH fields are known as symplectic vortices and their moduli space serves as a tool for defining invariants on symplectic manifolds with Hamiltonian actions, as discussed in  \cite{CGS00, Mundet00, Mundet-Tian, S16}. The well-known Hitchin's model~\cite{Hitchin87} arises when $\F=\Om^1(ad\P)$.

The investigation of removable singularities of YM fields and YMH fields began when Uhlenbeck established the fundamental compactness results of YM fields in \cite{U82,U85}. In \cite{U82'}, Uhlenbeck showed that a 4d YM connection with a point singularity is gauge equivalent to a smooth connection if its YM energy is finite. Later R{\aa}de \cite{Ra93} gave two different proofs of the removability of isolated singularities for 4-dimensional YM fields. Parker \cite{Parker1} extended this result to YMH fields with scalar and spinor fields.

In dimension 3, YM fields with finite energy in the entire space are trivial~\cite{Taub}. For dimensions greater than 4, there exist finite-energy YM fields with non-removable singularities~\cite{U82}. This suggests that finite YM energy is an inadequate criterion for removing point singularities in dimensions other than 4. A natural candidate in dimension $n$ is a finite $L^\frac{n}{2}$-norm of curvature, which is a conformally invariant functional. This was verified in dimension 3 by Sibner~\cite{S84} and in dimension greater than 4 by Sibner and Sibner~\cite{SS84}. They also extended this result to YMH fields~\cite{OSib}.

Point singularities of 2-dimensional YMH fields turned out to exhibit new features. In fact, codimension two singular sets of $n$-dimensional YMH fields are in general not removable due to possible non-trivial limiting holonomy supported at the singular points, see \cite{Smith,SS92,Ra94,Ra94'}. The analysis of higher dimensional singular sets of YMH fields becomes more complicated and often requires additional assumptions. For example, Tao and Tian~\cite{TT04} proved the removability of codimension 4 singularities of small energy stationary admissible YM fields. Smith and Uhlenbeck \cite{SU19} provided another proof of this result for YMH fields by using Kato inequalities for YM fields and elliptic regularity theory in Morrey spaces.

Note that all the above results on removable singularities of YMH fields require the coupling Higgs bundle being a vector bundle. When the fiber space is a Riemannian manifold, the equation of the Higgs field becomes non-linear and significantly affects the asymptotic behavior of YMH fields near singularities. Indeed, the study of singularities of harmonic maps~\cite{SU81, L85} clearly shows that new difficulties arise when the fiber space is curved. In this case, few results exists on the removability of singularities, even for isolated singularities. In our privious work \cite{CS21}, we established a sharp decay estimate for YMH fields with point singularities on a surface, leading to a theorem on removable isolated singularities for 2d YMH fields if the limiting holonomy is trivial. When the dimension exceeds 2, there is no limiting holonomy at isolated singularities, but the asymptotic behavior of the YMH fields and, in particular, the interaction between the connection and the Higgs field near the singular points remains unclear. We aim to establish decay estimates and removable singularity theorems for isolated singularities of 3d YMH fields in this paper and of higher dimensional YMH fields in another sequel.

\subsection{YMH fields with isolated singularities}
Now we  focus on 3d YMH fields with isolated singularities, where the fiber space is a compact Riemannian manifold. For local regularity theory, the curvature of the base manifold is not particularly important and we can set the base manifold to be flat for simplicity. Let $B_1\subset \Real^3$ denote the unit Euclidean ball and $B^*_1=B_1\setminus\{0\}$ denote the punctured ball. Let $\P$ be a trivial principal $G$-bundle over $B_1^*$ and $\F=\P\times_G N$ be the associated bundle with fiber modeled on a compact Riemannian manifold $N$. We denote the space of smooth sections by $\S=\Gamma(\F)$ , and the affine space of smooth connections on $\P$ by $\A$. Under a fixed trivialization of the bundles, we can identify a section $\Phi\in \S$ with a map $u:B_1^*\to N$ and a connection $A\in\A$ with a $\g$-valued 1-form. And the Higgs potential  is simply a gauge invariant smooth function $V: N\to \Real$.

Then the YMH functional (\ref{YMHF0}) becomes
\begin{equation}\label{eq-ymh1}
\mathscr{E}(A, u)=\int_{B_1^*}(|\n_A u|^2+|F_A|^2+V(u))dx,
\end{equation}
and a YMH field $(A,u)$ with isolated singularity at the origin is a solution of the following equation
\begin{equation}\label{eq:EL1}
	\begin{cases}
		D_A^*F_A=-u^*\n_Au,\\
		\n^*_A\n_A u=-\n V(u),
	\end{cases}
\end{equation}
defined in the punctured ball $B_1^*$.

Our main result gives an almost optimal decay estimate of 3d YMH fields near isolated singular points.

\begin{thm}\label{decay-YMH}
There exists an $\ep_0>0$ such that if $(A,u)$ is a YMH field  on $B^*_{R_0}\subset \Real^3$ with $R_0\leq 1$, which satisfies one of the following energy bounds
\begin{itemize}
\item[$(a)$] energy bound I:
\begin{equation}\label{Mor-c0}
	\sup_{B_{\rho(y)\subset B_{R_0}}}\frac{1}{\rho}\int_{B_{\rho}(y)}(|\n_A u|^2+|F_{A}|)dx\leq \ep^2_0,
\end{equation}	
\item[$(b)$] energy bound II:
\begin{equation}\label{Mor-c1}
\frac{1}{R_0}\int_{B_{R_0}}|\n_A u|^2dx+R_0^{\ka}\(\int_{B_{R_0}}(|F_A|^2+|\n V(u)|^2)dx\)^{\frac{1}{2}}\leq \ep^2_0,
\end{equation}
for some $\ka \in (0,\frac{1}{2})$;
\end{itemize}
then there is a $0<r_0\leq R_0/2$ such that $(A,u)$ satisfies
\begin{equation}\label{main-theorem-curvature}
  r^2\sqrt{|F_A(x)|^2+1}\leq C(\ep^2_0+r^2_0)\(\frac{r}{r_0}\)^{1+2\al},
\end{equation}
and
\begin{equation}\label{main-theorem-section}
  r\sqrt{|\n_A u(x)|^2 +1}\leq C\(\ep_0+r_0\)\(\frac{r}{r_0}\)^{\frac{8\al+1}{5}}
\end{equation}
for any $r=|x|\leq r_0$, where $\al=\sqrt{\frac{1}{4}-C(\ep^2_0+r^{2/3}_0)}$. Here the constant $C$ in \eqref{main-theorem-curvature} and \eqref{main-theorem-section} depends only on the geometry of $N$ in the case of energy bound I, and depends on the geometry of $N$ and $\ka$ for energy bound II.

\end{thm}

Here the energy bound I \eqref{Mor-c0} is a natural confomally invariant Morrey norm, while the energy bound II \eqref{Mor-c1} only requires an energy bound on a disk of fixed size $R_0$, which is more convenient for applications. Indeed, the second part of energy bound II, i.e. the bound on the curvature and the Higgs potential can be easily achieved provided the YMH energy of the YMH field is bounded and $R_0$ is sufficiently small.

As a by product, we also obtain the following decay estimates for pure Yang-Mills fields and pure harmonic maps in the vicinity of isolated singularities, as both are specific instances of YMH fields. There estimates are sharp and parallel to the well-known results for 4d YM fields~\cite{U82'} and 2d harmonic maps~\cite{SU81}.

\begin{thm}\label{decay-pureYM-and-PureHM}
\begin{itemize}
\item[$(1)$] There exists an $\ep_1>0$ such that if $A$ is a smooth Yang-Mills connection on $B^*_{R_0}\subset \Real^3$, which satisfies the energy bound
\[\sup_{B_{\rho(y)\subset B_{R_0}}}\(\frac{1}{\rho}\int_{B_{\rho}(y)}|F_A|dx\)\leq \ep^2_1,\]
then
\[r^2|F_A(x)|\leq C\ep^2_1 \(\frac{r}{R_0}\)^{2}\]
for any $r=|x|\leq R_0/2$;

\item[$(2)$] There exists an $\ep_2>0$ such that if $u$ is a smooth harmonic map $u$ on $B^*_{R_0}\subset \Real^3$, which satisfies the energy bound
\[\frac{1}{R_0}\int_{B_{R_0}}|\n u|^2dx\leq \ep^2_2,\]
then
\[r|\n u(x)|\leq C\ep_2\frac{r}{R_0}\]
for any $0<r\leq R_0/2$.
\end{itemize}
\end{thm}

As an application of Theorem \ref{decay-YMH}, we get the following theorem of removable isolated singularities of 3d YMH fields.

\begin{thm}\label{remov-iso-singu}
There exists an $\ep_0>0$ such that if $(A,u)$ is a smooth YMH field on $B^*_{1}\subset \Real^3$, which satisfies energy bound I \eqref{Mor-c0} or energy bound II \eqref{Mor-c1} with some $R_0\leq 1$, then $(A,u)$ extends to a smooth YMH field on $B_{1}$.
\end{thm}

In particular, the removable singularity theorem holds for 3d YMH fields with bounded conformally-invariant energy.
\begin{cor}\label{remov-iso-singu1}
Let $(A,u)$ be a smooth YMH field on $B^*_1\subset \Real^3$ with bounded conformally invariant energy
\begin{equation}\label{conf-en}
\int_{B_{1}}|\n_A u|^3dx+\int_{B_{1}}|F_A|^{\frac{3}{2}}dx<\infty.
\end{equation}
Then $(A,u)$ extends to a smooth YMH field on $B_{1}$.
\end{cor}

\begin{rem}
Theorem \ref{remov-iso-singu} and Corollary \ref{remov-iso-singu1}  generalizes the removable singularity theorems for 3d Yang-Mills fields in \cite{S84,SS84} and for 3d harmonic maps in \cite{L85} with finite conformally invariant $L^p$ energy. In fact, here we provide a new proof for these classical results.
\end{rem}

Here we outline the proof of the main Theorem \ref{decay-YMH}, as the rest results follows easily from Theorem \ref{decay-YMH}. The main idea of proof follows that of Smith-Uhlenbeck~\cite{SU19} which originated from R\aa{}de~\cite{Ra93}.

The key ingredient is a new improved Kato inequality for YMH fields $(A,u)$, which actually holds in all dimensions. In particular, in dimension 3 it has the following form:
\begin{align*}
|\n F_A|^2\geq& \frac{3}{2}|\n |F_A||^2-|u^*\n_A u|^2,\\	
|\n_A \n_A u|^2\geq &(\frac{3}{2}-\de)|\n |\n_Au||^2-C_\de(1+|F_A|^2)
\end{align*}
for $\de\in (0, \frac{1}{4})$, where $C_\de$ is a constant depending only on $\de$. The above Kato inequalities improve those derived by Smith and Uhlenbeck \cite{SU19} and are almost optimal in the sense that $\de$ can be arbitrarily small.

Equipped with the above Kato inequalities, we can derive a differential inequality for the YMH field by applying the Bochner formula. Actually, we find that the functions $(f, g):=\((|F_A|^2+1)^{\frac{1}{4}}, (|\n_A u|^2+1)^{\frac{1+2\de}{4}}\)$ satisfies
\begin{align}
\De f+C(|\n_A u|^2+f^2)f\geq 0,\label{eq0.1}\\
\De g+C(|F_A|^2+|\n_A u|^2+1)g\geq 0\label{eq0.2}.
\end{align}

Now suppose the YMH field satisfies the Morrey-type energy bound I, i.e. (\ref{Mor-c0}). It follows from standard $\ep$-regularity theorem that
\begin{equation}\label{eq-ep-reg}
  r^2(|\n_A u|^2+|F_A|)\le C\ep^2.
\end{equation}
To obtain the desired decay estimates for YMH fields from the coupled differential inequalities, we start by estimating the curvature $F_A$. By rewriting inequality \eqref{eq0.1} in cylindrical coordinates $(t=-\log|x|, \th)$, we get that the function $\bar{f}(t)=e^{-\frac12t}f(t)$ satisfies
\begin{equation}\label{ineq-f-bar}
  \p^2_t\bar{f}+\De_{\mathbb{S}^2}\bar{f}-\al^2\bar{f}\geq 0,
\end{equation}
where $\al=\sqrt{\frac{1}{4}-C(\ep^2_0+r^{2/3}_0)}$. But this is not sufficient to get the decay estimate of $|F_A|$ since $\al$ is strictly less than $1/2$. To overcome this difficulty, we need to invoke the broken Coulumb gauge constructed by Uhlenbeck~\cite{U82'} to show that $F_A$ actually enjoys an estimate stronger than (\ref{eq-ep-reg}). Namely, we have
\[e^{-2t}|F_A|(t,\th)\leq C\ep^2_0e^{-2\ka_0 t}\]
for a positive constant $\ka_0>0$, see Lemma~\ref{Bro-gauge}. This improved estimate together with the inequality (\ref{ineq-f-bar}) yields the desired estimate~(\ref{main-theorem-curvature}) for the curvature $F_A$. To prove the estimate of $\n_A u$, we combine the curvature estimate~(\ref{main-theorem-curvature}) and inequality \eqref{eq0.2} to show that the scaled function $\bar{g}(t)=e^{-\frac{1}{2}t}g(t)$ satisfies
\[\p^2_t \bar{g}+\De_{\mathbb{S}^2}\bar{g}-\al^2\bar{g}\geq -Ce^{-4\al (t-t_0)}\bar{g},\]
where $t_0=-\log r_0$. Then the decay estimate~(\ref{main-theorem-section}) follows from the comparison principal for ODEs. This finishes the proof of Theorem~\ref{decay-YMH} in the case of energy bound I.

For YMH fields satisfying energy bound II (\ref{Mor-c1}), the standard $\ep$-regularity theorem does not apply. Fortunately, we are still able to prove an improved $\ep$-regularity theorem (Theorem~\ref{ep-3D}) under energy bound II, which is in turn based on a new monotonicity inequality for YMH fields (see Lemma~\ref{mono-formu} below). The rest arguments are similar to the case of energy bound I.

\medskip

The rest of our paper is organized as follows. In Section \ref{s: pre}, we recall some preliminary results for YMH fields. In Section \ref{s:Ka-inq}, we derive the improved Kato inequalities for 1-forms and 2-forms on vector bundles, which yields an improved Kato inequality for 3D YMH fields. Next we prove our main Theorem \ref{decay-YMH} for YMH fields satisfying energy bound I in Section~\ref{s:ebound-I}  and for YMH fields satisfying energy bound II in Section~\ref{s:ebound-II}. The proof of Theorem \ref{decay-pureYM-and-PureHM} is given in Section~\ref{s:decay-pure}. Finally, we give the proof of Theorem \ref{remov-iso-singu} and Corollary \ref{remov-iso-singu1} in Section \ref{s: remov-singu}.

\medskip
\section{Preliminaries}\label{s: pre}
\subsection{Equations for YMH fields}
Let $G$ be a compact connected Lie group equipped with a bi-invariant metric and $\g$ be its Lie algebra. Let $(N,h)$ be a compact Riemannian manifold, which supports an left action of $G$ preserving the metric $h$. For any fixed $R>0$, let $B_R\subset \Real^n$ be the Euclidean ball with radius $R$ and $B^*_R=B_R\setminus\{0\}$ be the punctured ball. Let $\P$ be a trivial principal $G$-bundle over $B_R$ and $\F=\P\times_G N$ be the associated bundle with fiber $N$. Under a fixed trivialization, a section $\Phi$ of $\F$ can be identified with a map $u:B_R\to N$, while a connection $A\in\A$ is just a $\g$-valued 1-form.

For simplicity, let $\S:=C^\infty(B_R, N)$ denote the space of smooth sections of $\mathcal{F}$ and $\A=\Om^1(B_R,\g)$ denote the space of smooth connections. A connection $A\in \A$ induces an extrinsic derivative $D_A=d+A$ and the curvature of $A$ is defined by
\[F_A=D^2_A=dA+\frac{1}{2}[A,A].\]
The connection $A$ also induces a covariant derivative $\n_A=\n+A$ such that for any $u\in \S$ we have
\[\n_A u=\n u+A\cdot u,\]
where $\cdot$ denote the infinitesimal action of $\g$ on $N$.

Under a gauge transformation $s:B_R\to G$, a pair $(A,u)\in \A\times \S$ satisfies
$$s^{*}A=s^{-1}ds+s^{-1}As,\, s^{*}F_{A}=F_{s^{*}A}=s^{-1}F_{A}s,$$
and
$$s^*D_A u=D_{s^*A}s^*u=s^{-1}D_{A}u,\, V(s^*u)=V(u)$$
since the Higgs potential $V(u)$ is a gauge invariant. The Yang-Mills-Higgs functional \eqref{eq-ymh1} is invariant under the gauge transformation $s$, that is
$$\mathscr{E}(A,u)=\mathscr{E}(s^{*}A,s^{*}u).$$

A YMH field $(A,u)$ is a critical point of YMH functional $\mathscr{E}$, which satisfies the following equation
\begin{equation}\label{eq-YMH1}
	\begin{cases}
		D^{*}_{A}F_{A}=-u^{*}\n_{A}u,\\
		\n^{*}_{A}\n_{A}u=-\n V(u).
	\end{cases}
\end{equation}
We write equation \eqref{eq-YMH1} in a more explicit form. Recall that the infinitesimal action of $\g$ on $N$ is defined as follows. For $\forall c\in \g$, let $\varphi_{s}=\exp(s c): N\mapsto N$ be the $1$-parameter
group of isomorphism generated by $c$. Then
$$c\cdot y:=\frac{d}{ds}\Big|_{s=0}\varphi_s(y)=X_{c}(y),$$
where $X_c$ is a Killing vector field on $N$.
Similarly, $c$ acts on a vector field $Y\in \Ga(TN)$ by
$$c\cdot Y:= \frac{\nabla}{ds}\Big|_{s=0}(\varphi_{s})_{*}(Y)=\nabla_{Y} X_{c}.$$
where $\n$ is the Levi-Civita connection on $N$. Since the $G$-action preserves the metric $h$ on $N$,  $X_{c}$ is a Killing field, then $\nabla X_{c}$ is skew-symmetric, i.e.
$$h(\nabla _{Y}X_{c},Z)=-h(\nabla _{Z}X_{c},Y),\quad Y,\,Z \in \Ga(TN).$$
Then a direct calculation shows that equation \eqref{eq-YMH1} is equivalent to
\begin{equation}\label{eq-YMH2}
\begin{cases}
d^{*}dA+[A,dA] +[A,[A,A]]=-u^{*}(\n_{A}u),\\
\tau(u)-d^{*}A\cdot u+2A\cdot du+A\cdot(A\cdot u)=-\n V(u),
\end{cases}
\end{equation}
where $\tau(u)=\n^*\n u$ denote the tension field of map $u$.

For the purpose of PDE analysis, we also need an extrinsic form of \eqref{eq-YMH2}. First we recall that by the equivariant embedding theorem by Moore and Schlafly~\cite{MS}, there exists an isometric embedding
	$i:N\rightarrow\mathbb{R}^{K}$  and a representation $\rho:G\longrightarrow SO(K)$, such that $i(g.y)=\rho(g)i(y)$, for any $y\in N$ and $g\in G$. Under this representation, the Lie algebra $\mathfrak{g}$ corresponds to a sub-algebra of $\mathfrak{so}(K)$, i.e. the space of skew-symmetric
$K\times K$ matrices. Thus for any $c\in \mathfrak{g}$ and $y\in N \hookrightarrow \mathbb{R}^{K}$, the infinitesimal action of $c$ on $y$ is simply
$$
c\cdot y=X_{c}(y)=\rho(c) y.
$$
It follows that the action of $a$ on a vector field $Y\in \Ga(TN)$ is
$$
c\cdot Y=\nabla_{Y}X_{c}=(\rho(c)\cdot Y)^{\top}=\rho(c) Y-\Ga(y)(X_{c}, Y),
$$
where $\top$ denotes the projection from $\mathbb{R}^K$ to the tangent space of $N$ and $\Ga$ denotes the second fundamental for of $N$ in $\Real^K$.

Using these notations, we can rewrite equation \eqref{eq-YMH2} as
\begin{equation}\label{eq-YMH3}
	\begin{cases}
		d^{*}dA+[A,dA] +[A,[A,A]]=-u^{*}(\n_{A}u),\\
		\Delta u-d^{*}Au+2Adu+A(Au)=\Ga(u)(\n_Au,\n_Au)+\n V(u).
	\end{cases}
\end{equation}

\medskip
\subsection{Bochner formula for YMH fields}
Next we derive the Bochner formula for a YMH field $(A,u)$. Let $\{x^i\}$ be natural coordinates on $B_R$. For any $p\in B_R$, let $\{y^\al\}$ be normal coordinates at $u(p)\in N$. Then we have
\[\n_A u=\n u+A\cdot u=u^\al_idx^i\otimes\frac{\p}{\p y^\al},\]
where we denote the components
$$u^\al_i=\frac{\p y^\al(u)}{\p x^i}+(A_i)^\al_\beta y^\beta(u). $$

Let $R^N$ denote the curvature tensor of $N$. A simple calculation shows
\begin{equation}\label{bch1}
\begin{aligned}
(\De_A \n_A u)^\al_j=&\n_{A_i} \n_{A_i}u^\al_j=\n_{A_i}\n_{A_j} u^\al_i+\n_{A_i}(F_{ij}\cdot u)^\al\\
=&\n_{A_j}\n_{A_i}u^\al_i+R^{N}_{\ga zw\al}u^\ga_iu^z_ju^w_i
+(F_{ij}\cdot \n_{A_i}u)^\al\\
&+(\n_{A_i}F_{ij}\cdot u)^\al+(F_{ij}\cdot \n_{A_i}u)^\al\\
=&-(\n_A \n^*_A\n_A u)^\al_j+R^{N}_{\ga zw\al}u^\ga_iu^z_ju^w_i\\
&+2(F_{ij}\cdot \n_{A_i}u)^\al-(D^*_A F\cdot u)^\al_j.
\end{aligned}
\end{equation}
Since $(A,u)$ is a YMH field which satisfies \eqref{eq-YMH1}, we can derive from \eqref{bch1} that
\begin{equation}\label{bch2}
\begin{aligned}
\frac{1}{2}\De|\n_A u|^2=& |\n_A \n_A u|^2+|u^*(\n_Au)|^2\\
&+2F_A(\n_A u, \n_A u)+\n^2V(u)(\n_Au,\n_Au)\\
&+R^N\# \n_Au\# \n_Au\# \n_Au \#\n_Au\\
\geq &|\n_A \n_A u|^2+|u^*(\n_Au)|^2\\
&-C(1+|F_A|+|\n_A u|^2)|\n_A u|^2,		
\end{aligned}
\end{equation}
where the constant $C$ depends only on $\norm{\n^2 V}_{L^\infty}$ and $\norm{R^N}_{L^\infty}$.

For the curvature $F_A$, a direct computation gives
\begin{equation}\label{bch3}
	(-\n^*_A\n_A F_A)_{ij}=-(D_AD^*_AF_A)_{ij}+[(F_A)_{ik}, (F_A)_{kj}],
\end{equation}
where we have used the Bianchi identity $D_A F_A=0$. Hence by \eqref{eq-YMH1} and \eqref{bch3} we get
\begin{equation}\label{ineq2}
	\frac{1}{2}\De|F_A|^2\geq |\n_A F_A|^2+|F_A\cdot u|^2-C|F_A||\n_Au|^2-C|F_A|^3,
\end{equation}
where $C$ is a universal constant. In summary, we obtain

\begin{lem}\label{bochner}
There exists positive constants $a, b, c$ such that if $(A,u)$ is a YMH field on $B_R$, then the following holds
\begin{itemize}
\item [$(1)$] $|\n_A u|^2$ satisfies the differential inequality
\[\De |\n_A u|^2\geq -a(1+|F_A|)|\n_Au|^2-b|\n_A u|^4;\]
\item[$(2)$] $|F_A|$ satisfies the differential inequality
\[\De |F_A|\geq -a|F_A|^2-b|\n_A u|^2;\]
\item[$(3)$] Let $f=|\n_A u|^2+|F_A|$, we have
\[\De f\geq -c(1+f)f.\]
\end{itemize}
\end{lem}

\subsection{$\ep$-regularity theorems}

The following $\ep$-regularity theorem for YMH fields with energy bound I is standard. We state the theorem for general dimensions $n\ge 2$ and include a proof here for completeness.

\begin{lem}\label{ep-reg}
There exists an $\ep_0>0$ such that for any YMH field $(A,u)$ on $B_{R}$, satisfying energy bound I:
\begin{equation}\label{Mor-c}
\sup_{B_{\rho(y)\subset B_{R}}}\(\frac{1}{\rho^{n-2}}\int_{B_{\rho}(y)}|\n_A u|^2+|F_{A}|dx\)\leq \ep^2_0,
\end{equation}	
then we have
\[\sup_{B_{R/2}}R^2(|\n_A u|^2+|F_A|)\leq C_N\ep^2_0,\]
where $C_N$ is a constant depending only on the curvature of $N$.
\end{lem}
\begin{proof}
By Lemma \ref{bochner}, the function $f=|\n_A u|^2+|F_A|$	 satisfies
\[\De f\geq -c(1+f)f.\]
on $B_{R}$. It follows that the rescaled function $G(x)=R^2f(Rx)$ satisfies
\begin{equation}\label{eq:G}
\De G\geq -c(R^2+G)G
\end{equation}
on $B_1$.
	
Set $e(s)=(1-s)\sup_{B_s}G$, where $0\leq s\leq1$. Then there exists  $e_0$, $s_0<1$ and $x_0\in B_{s_0}$ such that
\[e_0=e(s_0)=\sup_{0\leq s\leq 1}e(s)=(1-s_0)G(x_0).\]
Let $\rho_0=\frac{1-s_0}{2}$. It is easy to see that
\[\rho_0\sup_{B_{\rho_0}(x_0)}G\leq e_0.\]
	
Now assume $e_0>1$. Set $\rho_1=\frac{\rho_0}{e_0}\leq \rho_0<1$ and consider the function $v(x)=\rho^2_1G(\rho_1x+x_0)$. It follows from \eqref{eq:G} that
\[\De v\geq -c(R^2\rho^2_1+v)v\]
on $B_1$ with
\[\sup_{B_1}v(x)\leq 1.\]
Then the standard Nash-Morse estimate (see Theorem 8.17 in \cite{GT} or Theorem 4.1 in \cite{HL})  yields
\[1=v(0)\leq C\(\int_{B_1}v(x)dx\)=\frac{C}{(R\rho_1)^{n-2}}\int_{B_{R\rho_1}(x_0)}f(x)dx\leq C\ep^2_0,\]
which is a contradiction, since $\ep_0$ can be chosen to be small enough.
	
Consequently, we must have $e_0\leq 1$. It follows
\[\sup_{B_s}G\leq \frac{1}{1-s}.\]
Taking $s=\frac{3}{4}$, we have that $\sup_{B_{3/4}}G\leq 4$ and by~\eqref{eq:G}
\[\De G\geq -cG,\]
in $B_{3/4}$. Then the Nash-Morse estimate gives
\[\sup_{B_{1/2}}G\leq C\int_{B_{3/4}}G(x)dx=C\frac{1}{(3/4R)^{n-2}}\int_{B_{3/4R}}f(x)dx\leq C\ep^2_0.\]
Rescaling back, we get the desired estimate for $f$ and the proof is finished.
\end{proof}

A direct corollary of Theorem \ref{ep-reg} is as follows.
\begin{cor}
There exists an $\ep_0>0$ such that for any YMH field $(A,u)$ on $B_{R}$, satisfying
\begin{equation}\label{inq1}
\int_{B_R}|\n_A u|^n+|F_{A}|^{n/2}dx\leq \ep^n_0,
\end{equation}	
then we have
\[\sup_{B_{R/2}}R^2(|\n_A u|^2+|F_A|)\leq C_N\ep^2_0\]
where $C_N$ is a constant depending only on the curvature of $N$.
\end{cor}

For later applications, we also need the following more general $\ep$-regularity theorem for 3-dimensional YMH fields.
\begin{cor}\label{ep-es1}
Let $B_R\subset \Real^3$ be a ball with radius $R\leq 1$. There exists an $\ep_0>0$ such that for any YMH field $(A,u)$ on $B_{R}$, satisfying
\begin{equation}\label{inq2}
\sup_{B_{\rho(y)\subset B_{R}}}\(\frac{1}{\rho}\int_{B_{\rho}(y)}|\n_A u|^2dx\)+R^{\ka}\norm{F_A}_{L^p(B_R)}\leq \ep^2_0,
\end{equation}	
for some $0\le \ka\le 2-\frac{3}{p}$ and $p\ge \frac{3}{2}$,
then we have
\[\sup_{B_{R/2}}\(R^2|\n_A u|^2+R^{\frac{3}{p}+\ka}|F_A|\)\leq C_N\ep^2_0,\]
where $C_N$ is a constant depending only on the curvature of $N$ and $p$.
\end{cor}
\begin{proof}
Since $0\le \ka\le 2-\frac{3}{p}$ and $p\ge \frac{3}{2}$, by H\"older inequality, we have for any $0<\rho\le R \le 1$,
\[ \frac{1}{\rho}\int_{B_\rho}|F_A|dx\le C\rho^{2-3/p}\|F_A\|_{L^p(B_\rho)}\le CR^{\ka}\norm{F_A}_{L^p(B_R)}.\]
So condition \eqref{inq2} implies \eqref{Mor-c} and it follows from Theorem \ref{ep-reg} that
\[\sup_{B_{3R/4}}R^2(|\n_A u|^2+|F_A|)\leq C\ep^2_0\]
where $C$ is a constant depending only on the curvature of $N$ and $p$.
	
To improve the estimate of $F_A$, note that by item (2) of Lemma \ref{bochner},
\[\De |F_A|\geq -a|F_A|^2-b|\n_A u|^2\geq -\frac{C}{R^2}|F_A|-\frac{C\ep^2_0}{R^2}.\]
By rescaling and letting $v(x)=|F_A|(Rx)$, we get
\[\De v\geq  -Cv-C\ep^2_0\]
on $B_{\frac{3}{4}}$. Then standard Nash-Morse estimate gives
\[\sup_{B_{1/2}}v\leq C\norm{v}_{L^p(B_{3/4})}+C\ep^2_0.\]
Rescaling back, we obtain
\[\sup_{B_{R/2}}|F_A|\leq \frac{C}{R^{\frac{3}{p}}}\norm{F_A}_{L^p(B_{3R/4})}+C\ep^2_0\leq C(\frac{\ep^2_0}{R^{\frac{3}{p}+\ka}}+\ep^2_0).\]
This gives the desired estimate of $|F_A|$ and finishes the proof.
\end{proof}

\section{Kato inequalities}\label{s:Ka-inq}
In this section, we establish almost-optimal Kato inequalities for YMH fields in general dimensions, which play a crucial roles in obtaining decay estimates for 3d YMH fields in the next section and in our sequel for n-dimensional YMH fields.

Let $B_1$ be the unit ball in $\Real^n$. Let $E$ be a trivial vector bundle over $B_1$, which is equipped with a smooth bundle metric $h$. Suppose that there is a connection $\n$ on $E$, which is compatible with the metric $h$, and we use $D$ to denote the exterior derivative induced by $\n$.
For any $k\in \mathbb{N}$, we denote $\Om^k(E)=\G(\wedge^kT^*B_1\otimes E)$ the space of $E$-valued $k$-forms on $B_1$. Let $\{x^i\}_1^n$ be the natural coordinates on $B_1$, $\{e_{\beta}\}_1^K$ be a moving frame on $E$.

\medskip
\subsection{Kato inequalities for 1-forms} Let $\om $ be a 1-form in $\Om^1(E)$. Then $\om$ and $\n \om$ can be written locally as
\[\om=\om^\beta_idx^i\otimes e_\beta\quad\text{and}\quad \n \om=\om^\beta_{ij}dx^i\otimes dx^j\otimes e_\beta.\]
Then for each $\beta$, $\(\om^\beta_{ij}\)$ is a matrix in $M_{n\times n}(\Real)$. We start with the following basic lemma.
\begin{lem}\label{orth1}
Let $B\in M_{n\times n}(\Real)$ be a matrix. Then $B$ has the following orthogonal decomposition with respect with the standard inner product
\[B=\breve{B}_{sym}+\frac{1}{n}\tr{B}\,I_{n\times n}+B_{anti},\]
where $B_{anti}=\frac{B-B^{T}}{2}$ is the antisymmetric part of $B$,  $B_{sym}=\frac{B+B^{T}}{2}$ is the symmetric part of $B$ and $\breve{B}_{sym}$ is the trace free part of $B_{sym}$.
\end{lem}

For the $1$-form $\om$, since locally
\begin{itemize}
\item $\n \om=\om^\beta_{ij}dx^i\otimes dx^j\otimes e_\beta$ corresponds to the matrix $ B=(\om^\beta_{ij})$,
\item $D\om=\sum_{i<j}(\om^\beta_{ij}-\om^\beta_{ji})dx^i\wedge dx^j\otimes e_\beta$ corresponds $B_{anti}$ with norm $|D\om|^2=\sum_{\beta}\sum_{i<j}(\om^\beta_{ij}-\om^\beta_{ji})^2$,
\item $D^*\om=-\tr B,$
\end{itemize}
then Lemma \ref{orth1} implies that $\n \om$ locally has the following orthogonal decomposition
\[\n \om =-\frac{1}{n}D^*\om I_{n\times n}+\frac{1}{\sqrt{2}}D\om+{\breve{\n \om}}_{sym}.\]

\begin{lem}\label{Kato-ineq-1}
Let $\om$ be a $C^1$-smooth $1$-form in $\Om^1(E)$ with $n\geq 2$. Then we have
\[\frac{n}{n-1}\Big|\n|\om|-\frac{D^*\om}{n}\frac{\om}{|\om|}-\frac{1}{2}D\om\cdot\frac{\om}{|\om|}\Big|^2\leq |\n \om|^2-\frac{(D^*\om)^2}{n}-\frac{1}{2}|D\om|^2.\]
Furthermore, if $\om$ is harmonic, i.e. $D\om=0$ and $D^*\om=0$, then
\[\frac{n}{n-1}|\n|\om||^2\leq |\n \om|^2.\]
\end{lem}
\begin{proof}
For any $C^1$-smooth 1-form $X\in \Om^1(E)$, a simple computation shows
\[X\cdot\n|\om||\om|=X\cdot \n(\frac{1}{2}|\om|^2)=\<\n \om, X\otimes \om\>.\]
On the other hand, Lemma \ref{orth1} yields the following orthogonal decomposition for $\n \om$ and $X\otimes \om$ respectively
\begin{itemize}
	\item $\n \om =-\frac{1}{n}D^*\om I_{n\times n}+\frac{1}{\sqrt{2}}D\om+B_1$,
	\item $X\otimes \om=\frac{1}{n}\<X,\om\>I_{n\times n}+\frac{1}{\sqrt{2}}X\wedge \om+B_2.$
\end{itemize}
Then we have
\begin{align*}
	\<\n \om, X\otimes \om\>=&-\frac{1}{n}D^*\om<X,\om>+\frac{1}{2}\<D\om,X\wedge\om\>+\<B_1,B_2\>.
\end{align*}
Consequently, we can derive from this formula that
\begin{equation}\label{eq-1}
\begin{aligned}
	&\Big|X\cdot\n|\om||\om|-\frac{1}{n} D^*\om<X,\om>-\frac{1}{2}\<D\om,X\wedge\om\>\Big|^2\\
	=&\<B_1,B_2\>^2\leq |B_1|^2|B_2|^2\\
	\leq &\frac{n-1}{n}\(|\n w|^2-\frac{1}{n}(D^*\om)^2-\frac{1}{2}|D\om|^2\)|X|^2|\om|^2.
\end{aligned}
\end{equation}
Here we have used the fact that
\begin{align*}
	|B_1|^2=&|\n w|^2-\frac{1}{n}(D^*\om)^2-\frac{1}{2}|D\om|^2;\\
	|B_2|^2=&|X|^2|\om|^2-\frac{1}{n}\<X,\om\>^2-\frac{1}{2}(|X|^2|\om|^2-\<X,\om\>^2)\\
	=&\frac{1}{2}|X|^2|\om|^2+\frac{n-2}{2n}\<X,\om\>^2
	\leq \frac{n-1}{n}|X|^2|\om|^2.
\end{align*}
Now the lemma follows from \eqref{eq-1} and the arbitrariness of $X$.
\end{proof}

\medskip
\subsection{Kato inequalities for 2-forms}
A $C^1$-smooth 2-form $F\in \Om^2(E)$ can be written locally by
\[F=\frac{1}{\sqrt{2!}}F^\beta_{ij}dx^i\wedge dx^j\otimes e_\beta,\]
 where
\[F^\beta_{ij}=-F^\beta_{ji}\]
for each $\beta\in \{1,2,\cdots, K\}$. Thus we can view $F$ as an antisymmetric matrix $(F^\beta_{ij})$.

Let $\Om$ be an another 2-form $(\Om^\beta_{ij})$. The inner product of $F$ and $\Om$ is defined by
\[\<F,\Om\>=\frac{1}{2}F_{ij}\Om_{ij}=\sum_{\beta}\sum_{i<j}F^\beta_{ij}\Om^\beta_{ij}.\]
We also have
\begin{itemize}
\item $\n F=\frac{1}{\sqrt{2}}F^\beta_{ij,k}dx^i\otimes dx^j\otimes dx^k\otimes e_\beta$ with $F^\beta_{ij,k}=F^\beta_{ji,k}$ and
\[|\n F|^2=\frac{1}{2}F_{ij,k}\cdot F_{ij,k}=\sum_\beta\sum_{i<j,k}F^\beta_{ij,k}\cdot F^\beta_{ij,k},\]
\item $DF=\frac{1}{\sqrt{3!}}(F^\beta_{ij,k}+F^\beta_{jk,i}+F^\beta_{ki,j})dx^i\wedge dx^i\wedge dx^k\otimes e_\beta$ and
\begin{align*}
|DF|^2=&\frac{1}{3!}\sum_{\beta}\sum_{i,j,k}(F^\beta_{ij,k}+F^\beta_{jk,i}+F^\beta_{ki,j})^2\\
=&\sum_\beta\sum_{i<j<k}(F^\beta_{ij,k}+F^\beta_{jk,i}+F^\beta_{ki,j})^2,
\end{align*}
\item $D^*F=\sum_{j\neq i}F^\beta_{ij,j}dx^i\otimes e_\beta$.
\end{itemize}

The following result for 2-forms is analogous to Lemma \ref{orth1} for 1-forms.
\begin{lem}\label{oth2}
Let $F$ be a $C^1$-smooth 2-form in $\Om^2(E)$. Then $\n F$ has the following orthonogal decomposition
\begin{equation}\label{oth}
\n F=\frac{1}{(n-1)\sqrt{2}}\de_{jk}(i)(D^*F)_idx^i\otimes dx^j\otimes dx^k+\frac{1}{\sqrt{3}}DF+\breve{F},
\end{equation}
where $\breve{F}$ satisfies
\[\sum_{j\neq i}\breve{F}_{ij,j}=0 \,\,\text{and}\,\, \breve{F}_{ii,k}=0;\]
 the $\(\de_{jk}(i)\)$ satisfies $\de_{jk}(i)=1$ for $j=k\neq i$ and $d_{jk}(i)=0$ for other cases.
\end{lem}
\begin{proof}
By definition, $\breve{F}=\n F-\frac{1}{n-1}\de_{jk}(i)(D^*F)_idx^i\otimes dx^j\otimes dx^k-\frac{1}{\sqrt{3}}DF$. Then a simple calculation gives
\begin{align*}
\sum_{j\neq i}\breve{F}_{ij,j}=&\sum_{j\neq i}F^\beta_{ij,j}e_\beta-\frac{1}{n-1}\(D^*F\)^\beta_ie_\beta\sum_{j\neq i}1=0,\\
\breve{F}_{ii,j}=&F^\beta_{ii,j}=0.
\end{align*}

To show $\n F$ has the orthogonal decomposition \eqref{oth}, we only need to show
\begin{itemize}
	\item[$(1)$] $\<\frac{1}{n-1}\de_{jk}(i)(D^*F)_idx^i\otimes dx^j\otimes dx^k, DF\>=0$,
	\item[$(2)$] $\<\n F, \frac{1}{(n-1)\sqrt{2}}\de_{jk}(i)(D^*F)_idx^i\otimes dx^j\otimes dx^k\>=\frac{1}{2(n-1)}|D^*F|^2$;
	\item[$(3)$] $\<\n F, \frac{1}{\sqrt{3}}DF\>=\frac{1}{3}|DF|^2$.
\end{itemize}
The identity (1) and (2) follows from a straight forward calculation and we only give the proof of (3). For simplicity, we set $C^\beta_{ijk}=F^\beta_{ij,k}+F^\beta_{jk,i}+F^\beta_{ki,j}$, then $C^\beta_{ijk}=C^\beta_{jki}=C^\beta_{kij}$ for any $i,j,k$. On the other hand, we have
\begin{align*}
\<\n F,DF\>=&\frac{1}{2\sqrt{3}}\sum_{\beta}\sum_{i,j,k}\<F^\beta_{ij,k}, C^\beta_{ijk}\>\\
=&\frac{1}{2\sqrt{3}}\sum_{\beta}\sum_{i,j,k}\<F^\beta_{jk,i}, C^\beta_{jki}\>\\
=&\frac{1}{2\sqrt{3}}\sum_{\beta}\sum_{i,j,k}\<F^\beta_{ki,j}, C^\beta_{kij}\>.
\end{align*}
Thus, the fact that $C^\beta_{ijk}=C^\beta_{jki}=C^\beta_{kij}$ yields
\[3\<\n F, DF\>=\frac{1}{2\sqrt{3}}\sum_{\beta}\sum_{i,j,k}(C^\beta_{ki,j})^2=\sqrt{3}|DF|^2.\]
\end{proof}

Therefore, we can apply Lemma \ref{oth2} to obtain the following improved kato inequalities for 2-forms.

\begin{lem}\label{Kato-ineq-2}
Let $F$ be a $C^1$-smooth 2-form in $\Om^2(E)$ with $n\geq 3$. Then we have
\begin{itemize}
\item[$(1)$] If $F$ is a harmonic 2-form, i.e. $D^*F=0$ and $DF=0$, then
\begin{equation*}
\begin{cases}
\frac{3}{2}|\n |F||^2\leq |\n F|^2 \quad &n=3;\\[1ex]
\frac{n-1}{n-2}|\n |F||^2\leq |\n F|^2 \quad &n\geq 4.
\end{cases}	
\end{equation*}
\item[$(2)$]  If $F$ is closed, then
\begin{equation*}
\begin{cases}
\frac{3}{2}\Big|\n |F|-\frac{1}{2\sqrt{2}}\<D^*F,\frac{F}{|F|}\>\Big|^2\leq |\n F|^2-\frac{1}{4}|D^*F|^2 \quad &n=3,\\[1ex]
\frac{n-1}{n-2}\Big|\n |F|-\frac{1}{(n-1)\sqrt{2}}\<D^*F,\frac{F}{|F|}\>\Big|^2\leq |\n F|^2-\frac{1}{2(n-1)}|D^*F|^2 \quad &n\geq 4.
\end{cases}	
\end{equation*}
\end{itemize}
\end{lem}

\begin{proof}
For any $C^1$-smooth 1-form $X\in \Om^1(E)$, we have
\begin{align*}
X\cdot\n|F||F|=\<\n F, F\otimes X\>.
\end{align*}
By Lemma \ref{oth2}, $\n F$ has the orthonormal decomposition
	\[\n F=\frac{1}{(n-1)\sqrt{2}}\de_{jk}(i)(D^*F)_idx^i\otimes dx^j\otimes dx^k+\frac{1}{\sqrt{3}}DF+\breve{F}.\]
Similarly, we can also show that $F\otimes X$ has the following orthonogal decomposition
\[F\otimes X=\frac{1}{n-1}\de_{jk}(i)\<F,X\>_idx^i\otimes dx^j\otimes dx^k+\frac{1}{\sqrt{3}}F\wedge X+\breve{B}.\]
	
It follows that
\begin{align*}
\<\n F, F\otimes X\>
=&\frac{1}{(n-1)\sqrt{2}}D^*F\cdot\<F,X\>+\frac{1}{3}\<DF,F\wedge X\>+\<\breve{F}, \breve{B}\>,
\end{align*}
where
\begin{align*}
	|\breve{F}|^2=&|\n F|^2-\frac{1}{2(n-1)}|D^*F|^2-\frac{1}{3}|DF|^2;\\
	|\breve{B}|^2=&|F|^2|X|^2-\frac{1}{n-1}|\<F,X\>|^2-\frac{1}{3}|F\wedge X|^2\\
	=&|F|^2|X|^2-\frac{1}{n-1}|\<F,X\>|^2-\frac{1}{3}(|F|^2|X|^2-|\<F,X\>|^2)\\
	=&\frac{2}{3}|F|^2|X|^2+\frac{n-4}{3(n-1)}|\<F,X\>|^2\\
	\leq &\frac{n-2}{n-1}|F|^2|X|^2
\end{align*}
if $n\geq 4$. While for $n=3$ we have
\[|\breve{B}|^2\leq \frac{2}{3}|F|^2|X|^2.\]

Consequently, we get
\begin{align*}
&\Big|X\cdot\n|F||F|-\frac{1}{(n-1)\sqrt{2}}D^*F\cdot\<F,X\>+\frac{1}{3}\<DF,F\wedge X\>\Big|^2\leq |\breve{F}|^2|\breve{B}|^2\\
\leq &
\begin{cases}
\frac{2}{3}|\breve{F}|^2|F|^2|X|^2	 \quad &n=3,\\[1ex]
\frac{n-2}{n-1}|\breve{F}|^2|F|^2|X|^2  \quad &n\geq 4.
\end{cases}	
\end{align*}
Since the above inequality holds for all $X$, the lemma follows.
\end{proof}

In the 3-dimensional case, we also have the following useful Kato inequality.
\begin{lem}[Proposition 4.17 in \cite{D23}]\label{Kato-ineq-3}
Let $F$ be a $C^1$ smooth closed 2-form in $\Om^2(E)$ with $n=3$. Then we have
\[\frac{3}{2}|\n |F||^2\leq |\n F|^2+|D^*F|^2.\]

Similarly, if $\om$ is a $C^1$ smooth coclosed 1-form in $\Om^1(E)$ with $n=3$, then we have
\[\frac{3}{2}|\n |\om||^2\leq |\n \om|^2+|D\om|^2.\]
\end{lem}

\medskip
\subsection{Improved Kato inequalities for YMH fields}
Now we can apply Lemma \ref{Kato-ineq-1}-\ref{Kato-ineq-3} to obtain the following almost optimal Kato inequalities for YMH fields.

\begin{thm}\label{Kato-ineq-4}
Let $B_1\times N\to B_1$ be a trivial fiber bundle with compact structure group $G$, where $N$ is a compact Riemannian manifold. Suppose that $(A,u)$ is a smooth YMH field on $B^*_{2R_0}$, then for any $\de>0$, there exists a constant $C(n, \de)>0$ such that we have
\begin{itemize}
\item[$(1)$] the section $u$ satisfies
\begin{align}
|\n_A \n_A u|^2
\geq &(\frac{n}{n-1}-\de)|\n |\n_Au||^2-C(n, \de)(1+|F_A|^2);\label{Kato-ineq-4-1}
\end{align}
\item[$(2)$] the curvature $F_A$ of $A$ satisfies
\begin{equation}\label{Kato-ineq-4-2}
\begin{cases}
|\n F_A|^2\geq \frac{3}{2}|\n |F_A||^2-|u^*(\n_A u)|^2 \quad &n=3,\\[1ex]
|\n F_A|^2\geq (\frac{n-1}{n-2}-\de)|\n |F||^2-C(n, \de)|u^*(\n_A u)|^2 \quad &n\geq 4.
\end{cases}	
\end{equation}
\end{itemize}
\end{thm}
\begin{proof}
 Since $D_A u$ satisfies
\[D^*_A D_A u=\n V(u),\quad D_AD_A u=F_A\cdot u,\]
we can apply Lemma \ref{Kato-ineq-1} to get
\begin{align*}
	|\n_A \n_A u|^2\geq &(\frac{n}{n-1}-\de)|\n |\n_Au||^2-C(n, \de)(|\n V(u)|^2+|F_A\cdot u|^2)\\
	\geq &(\frac{n}{n-1}-\de)|\n |\n_Au||^2-C(n, \de)(1+|F_A|^2),
\end{align*}
for any $\de>0$, where have used the fact that $N$ is compact to show
\[|\n V(u)|+|u|\leq C.\]

Next, by using the fact that $F_A$ satisfies the following equation
\[D^*_A F_A=u^*(\n_A u),\quad D_AF_A=0,\]
we can apply Lemmas \ref{Kato-ineq-2}-\ref{Kato-ineq-3} to get the desired inequality \eqref{Kato-ineq-4-2}.
\end{proof}

\subsection{Differential inequalities of 3d YMH fields}

By applying the above Kato inequalities, we can get the differential inequalities for 3d YMH fields, which plays a key role in our proof of the main theorems.

\begin{lem}\label{diff-ineq-YMH}
Suppose $(A,u)$ is a smooth YMH field on $B^*_{R_0}$. Then for any $\de\in (0,\frac{1}{2})$, there exists a universal constant $C$ and a constant $C_\de>0$ such that the function $(f,g)=((|F_A|^2+1)^{1/4},(|\n_A u|^2+1)^{\frac{1}{2}(\frac{1}{2}+\de)})$ satisfies
\begin{align}
&\De f+C(|\n_A u|^2+f^2)f\geq 0,\label{eq1}\\
&\De g+C_\de(|F_A|^2+|\n_A u|^2+1)g\geq 0\label{eq3}
\end{align}
on $B^*_{R_0}$.

Or equivalently, in cylindrical coordinates $(t=-\log r,\th)$, we have
\begin{align}
&\p^2_tf-\p_t f+\De_{\mathbb{S}^2}f+Ce^{-2t}(|\n_A u|^2+f^2)f\geq 0,\label{eq2}\\
&\p^2_tg-\p_t g+\De_{\mathbb{S}^2}g+C_\de e^{-2t}(|F_A|^2+|\n_A u|^2+1)g\geq 0.\label{eq4}
\end{align}
\end{lem}
\begin{proof}
First we derive the differential inequality \eqref{eq1} and hence \eqref{eq2}.

Let $p>0$. For any smooth function $f$ on $B^*_{2R_0}$, we have
\begin{equation}\label{eq-f}
\De f^p=pf^{p-2}\(\frac{1}{2}\De f^2+\frac{p-2}{4}\frac{|\n f^2|^2}{f^2}\).	
\end{equation}
By letting $f=\sqrt{|F_A|^2+1}$ in equation \eqref{eq-f} and $p\in (0,2)$, we get
\begin{align*}
\De f^p=&pf^{p-2}\(\frac{1}{2}\De |F_A|^2+\frac{p-2}{4}\frac{|\n |F_A|^2|^2}{|F_A|^2+1}\)\\
\geq& pf^{p-2}(\frac{1}{2}\De |F_A|^2+(p-2)|\n |F_A||^2).
\end{align*}

Then the Bochner formula \eqref{ineq2} for $F_A$ and the Kato inequality \eqref{Kato-ineq-4-2} Lemma \ref{Kato-ineq-4} together implies
\begin{align*}
\frac{1}{2}\De |F_A|^2 \geq &|\n_A F_A|^2-C|F_A||\n_Au|^2-C(1+|F_A|)|F_A|^2\\
\geq &\frac{3}{2}|\n |F_A||^2-C(1+|F_A|)(|\n_Au|^2+|F_A|^2).
\end{align*}
Therefore, by choosing $\frac{1}{2}\leq p<2$, we get
\begin{align*}
\De f^p\geq &pf^{p-2}(\frac{3}{2}+p-2)|\n |F_A||^2-Cpf^{p-2}(1+|F_A|)(|\n_Au|^2+|F_A|^2)\\
\geq & -Cf^{p-1}(|\n_A u|^2+f^2)\\
\geq &-C(|\n_A u|^2+f)f^p,
\end{align*}
where we have used the facts that $|F_A|+1\leq \sqrt{2}f$ and $f\geq 1$. Then inequality \eqref{eq1} follows by taking $p=\frac{1}{2}$.

Next we derive the differential inequality \eqref{eq3} and hence \eqref{eq4}.

Replacing $f$ with $g=\sqrt{|\n_A u|^2+1}$ in formula \eqref{eq-f} and $p\in (0,2)$, we obtain
\[\De g^p\geq pg^{p-2}(\frac{1}{2}\De |\n_A u|^2+(p-2)|\n |\n_A u||^2).\]
For any $\de>0$, the Bochner formula \eqref{bch2} and the Kato inequality \eqref{Kato-ineq-4-1} in Lemma \ref{Kato-ineq-4} together gives
\begin{align*}
\frac{1}{2}\De|\n_A u|^2
\geq &|\n_A \n_A u|^2+|u^*(\n_Au)|^2-C(1+|F_A|+|\n_A u|^2)|\n_A u|^2\\
\geq &(\frac{3}{2}-\de)|\n |\n_Au||^2-C_{\de}(1+|F_A|^2+|\n_A u|^2)(|\n_A u|^2+1)\\
=&(\frac{3}{2}-\de)|\n |\n_Au||^2-C_{\de}(|F_A|^2+g^2)g^2.
\end{align*}

Therefore, by choosing $2>p\geq \frac{1}{2}+\de$ for any small $\de>0$, we get
\begin{align*}
\De g^p\geq &pg^{p-2}(p-\de-1/2)|\n |\n_A u||^2-C_{\de}pg^{p-2}(|F_A|^2+g^2)g^2\\
\geq & -C_{\de}(|F_A|^2+g^2)g^p.
\end{align*}
Taking $p=\frac{1}{2}+\de$, we get the desired inequality \eqref{eq3}.
\end{proof}

\medskip
\section{Decay estimates of 3d YMH fields: energy bound I}\label{s:ebound-I}

In this section, we derive the desired decay estimates for 3d YMH fields and prove the main Theorem~\ref{decay-YMH} under energy bound I \eqref{Mor-c0}.

\subsection{Improved regularity of $F_A$}

For 3d YMH fields satisfying energy bound I, we first observe that the curvature $F_A$ actually satisfies an estimate that is stronger than the standard $\ep$-regularity theorem (Theorem~\ref{ep-reg}). For this purpose, we first recall the broken Coulomb gauge constructed by Uhlenbeck in \cite{U82'} (see also \cite{S84}). For any fixed number $1<\tau<\frac{5}{4}$, let
\begin{align*}
	\Om_i=&\{x|\,\, \tau^{-i-1}<r\leq \tau^{-i} \},\\
	S_i=&\{x|\,\, r=\tau^{-i}\}.
\end{align*}
Then we have
\[B^*_{1}=\cup_{i=0}^\infty \Om_i.\]
\begin{lem}\label{Bro-gauge}
There exists an $\ep'_0>0$ such that for any connection $A$ satisfying $r^2|F_A|(x)\leq \ep'_0$ for all $x\in B^*_{1}$, there exists a continuous gauge transformation on $B^*_1$ such that the following properties hold for any $i\geq 0$:
\begin{itemize}
\item[$(1)$] $d^*A^i=0$ in $\Om^i$, where we set $A^i=A|_{\Om_i}$;
\item[$(2)$] $A^i_\th |_{S_i}=A^{i-1}_\th |_{S_i}$;
\item[$(3)$] $d^*_\th A^i_\th=0$ on $S_i$ and $S_{i+1}$;
\item[$(4)$] $\int_{S_i} A^i_rd\th=\int_{S_{i+1}}A^i_rd\th=0$;
\item[$(5)$] $|A_i|(x)\leq C\ep'_0\tau^{i+1}$;
\item[$(6)$] $\int_{\Om_i}|A^i|^2dx\leq \frac{1}{(\mu-C\ep'_0)\tau^{2(i+1)}}\int_{\Om_i}|F_{A^i}|^2dx$, where $\mu\geq 2$;
\item[$(7)$] $\int_{S_0}|A^0_\th|^2dV_{S_0}\leq \frac{1}{(\nu-C\ep'_0)}\int_{S_0}|F_A|^2dV_{S_0}$,
where $\nu=2$ is the first eigenvalue of Hodge Laplace operator for coclosed 1-form on $\U^2$.
\end{itemize}
Here we set $A=A_rdr+A_\th$, where $(r,\th)$ is the polar coordinates for $B^*_1$.
\end{lem}

\begin{lem}\label{decay-F1}
Assume that $(A,u)$ is a smooth YMH field on $B^*_{R_0}$, which satisfies
\begin{equation}\label{es-YMH*}
	r^2(|\n_A u|^2+|F_A|)(x)\leq \ep'_0.
\end{equation}	
for any $x\in B^*_{R_0}$, where $\ep'_0$ is given in Lemma \ref{Bro-gauge}. Then there exists a positive constant $\kappa_0\in [\frac{1}{16},1)$ such that
\begin{equation}\label{eq:lem4.2}
  r^2|F_A|\leq C\ep'_0\(\frac{r}{R_0}\)^{2\kappa_0}
\end{equation}
for any $x\in B^*_{R_0/2}$.
\end{lem}
\begin{proof}
The proof is divided into two steps.\

\medskip
\noindent\emph{Step 1.}
First we give an integral estimate of $F_A$ by using the broken Coulomb gauge.

Given any $R\le R_0$, $(A,u)$ restricts to a smooth YMH field on $B^*_{R}$. Rescaling by $R$, we may assume that $A$ satisfies the equation
\begin{equation}\label{eq01}
  D^*_A F_A=-R^2u^*(\n_A u)
\end{equation}
on the unit disc $B_1^*$, and \eqref{es-YMH*} holds for all $x\in B^*_{1}$. By Lemma \ref{Bro-gauge}, there exists a broken Coulomb gauge such that $A$ satisfies properties $(1)$-$(7)$.  By definition, in each annulus $\Om^i$, we have
$\tau^{-i-1}<r\leq \tau^{-i}.$
It follows from property $(6)$ that
\[(6^\prime) \quad \int_{\Om_i}r^\al|A^i|^2dx\leq \frac{\tau^\al}{(2-C\ep'_0)\tau^{2(i+1)}}\int_{\Om_i}r^{\al}|F_A|^2dx.\]
for any $\al\in (1,3)$.

On the other hand, since $F_{A^i}=dA^i+\frac{1}{2}[A^i,A^i]$, integration by parts gives
\begin{align*}
\int_{\Om_i}r^\al|F_{A^i}|^2dx=&\int_{\Om_i}\<dA^i+\frac{1}{2}[A^i,A^i],r^\al F_{A^i}\>dx\\
=&\int_{\Om_i}\<D_{A^i}A^i-\frac{1}{2}[A^i,A^i],r^\al F_{A^i}\>dx\\
=&\int_{\Om_i}\<A^i,D^*_{A^i}(r^\al F_{A^i})\>dx-\frac{1}{2}\int_{\Om_i}\<[A^i,A^i],r^\al F_{A^i}\>dx\\
&+(\int_{S_{i}}-\int_{S_{i+1}})r^\al(A^i_\th\wedge *F_{A^i})\\
\end{align*}
Inserting in equation \eqref{eq01}, we get
\begin{equation}\label{eq02}
  \begin{aligned}
\int_{\Om_i}r^\al|F_{A^i}|^2dx=&\int_{\Om_i}r^\al\<A^i,R^2u^*(\n_A u)\>dx-\al\int_{\Om_i}r^{\al-1}\<A^i,*(dr\wedge*F_{A^i})\>dx\\
&-\frac{1}{2}\int_{\Om_i}r^\al \<[A^i,A^i],F_{A^i}\>dx+(\int_{S_{i}}-\int_{S_{i+1}})r^\al(A^i_\th\wedge *F_{A^i})\\
=:&I+II+III+(\int_{S_{i}}-\int_{S_{i+1}})r^\al(A^i_\th\wedge *F_{A^i}).
\end{aligned}
\end{equation}

Next we estimate the above three terms $I$-$III$ respectively.  Using H\"older inequality and $(6^\prime)$ above, we have
\begin{align*}
|I|= &|\int_{\Om_i}r^\al\<A^i,R^2u^*(\n_A u)\>dx|\\
\leq &\ep\int_{\Om_i}r^\al|A^i|^2dx+C\ep^{-1}R^4\int_{\Om_i}r^\al|\n_A u|^2dx\\
\leq &\ep\frac{\tau^\al}{(2-C\ep'_0)\tau^{2(i+1)}}\int_{\Om_i}r^\al|F_{A^i}|^2dx+C\ep^{-1}R^4\int_{\Om_i}r^\al|\n_A u|^2dx,
\end{align*}
For the second term, we have
\begin{align*}
|II|\leq &\al\int_{\Om_i}|x|^{\al}|A^i||F_A|r^{-1}dx\\
\leq &\al\tau^{i+1}\(\int_{\Om_i}r^{\al}|A^i|^2dx\)^{1/2}\(\int_{\Om_i}r^{\al}|F_A|^2dx\)^{1/2}\\
\leq &\al(\frac{\tau^{\al}}{2-C\ep'_0})^\frac{1}{2}\int_{\Om_i}r^{\al}|F_A|^2dx,
\end{align*}
For the third term, using estimates $(5)$ in Lemma \ref{Bro-gauge}, we get
\begin{align*}
|III|\leq & \frac{1}{2}\norm{A^i}_{L^\infty}\(\int_{\Om_i}r^{\al}|A^i|^2dx\)^{1/2}\(\int_{\Om_i}r^{\al}|F_A|^2dx\)^{1/2}\\
\leq & C\ep'_0 \tau^{i+1}\(\int_{\Om_i}r^{\al}|A^i|^2dx\)^{1/2}\(\int_{\Om_i}r^{\al}|F_A|^2dx\)^{1/2}\\
\leq & C\ep'_0\(\frac{\tau^{\al}}{2-C\ep'_0}\)^\frac{1}{2}\int_{\Om_i}r^{\al}|F_A|^2dx.
\end{align*}
Since $1<\tau<\frac{5}{4}$, we can take $\al$ closed to $1$, $\ep_0$ and $\ep$ small enough such that
\[1-(C\ep'_0+\al)(\frac{\tau^{\al}}{2-C\ep'_0})^\frac{1}{2}-\frac{\ep\tau^\al}{(2-C\ep'_0)\tau^{2(i+1)}}\geq\frac{1}{8}.\]
Putting the above estimates into \eqref{eq02}, we arrive at
\begin{equation}\label{eq03}
\frac{1}{8}\int_{\Om_i}r^{\al}|F_A|^2dx\leq Cr^4\int_{\Om_i}r^\al|\n_A u|^2dx+(\int_{S_{i}}-\int_{S_{i+1}})r^\al(A^i_\th\wedge *F_{A^i}).
\end{equation}

Summing up \eqref{eq03} for all $i\ge 0$ and noting that
\[\limsup_{i\to \infty}\int_{S_i}r^\al|A^i_\th\wedge *F_{A^i}|dV_{S_i}\leq C(\ep'_0)^2\lim_{i\to \infty}\tau^{-i\al+1}=0,\]
we obtain
\begin{equation}\label{eq05}
  \int_{B_{1}}r^\al|F_A|^2dx\leq CR^4\int_{B_{1}}r^\al|\n_A u|^2dx+\frac{8}{(2-C\ep'_0)}\int_{\p B_1}|F_A|^2dV_{\p B_1}.
\end{equation}

Now rescaling back to $B_R$, inequality \eqref{eq05} becomes
\begin{equation}\label{eq04}
  \int_{B_{R}}r^\al|F_A|^2dx\leq CR^2\int_{B_{R}}r^\al|\n_A u|^2dx+CR\int_{\p B_R}r^\al|F_A|^2dV_{\p B_R}.
\end{equation}
If we set $I(R)=\int_{B_{R}}r^\al|F_A|^2dx$, inequality \eqref{eq04} and \eqref{es-YMH*} together yields
\[I(R)\leq C\ep'_0R^{\al+3}+bRI^\prime(R)\]
where $b=\frac{8}{(2-C\ep'_0)}$ is independent of $\al$. Solving this differential inequality, we find
\begin{equation}\label{eq:integral}
\begin{aligned}
I(R)\leq &C\ep'_0 (R^{3+\al-1/b}_0+\ep'_0R^{\al-1-1/b}_0)r^{1/b}\\
\leq &C\ep'_0(R^4_0+\ep'_0)R^{\al-1-1/b}_0R^{1/b}.
\end{aligned}
\end{equation}

\medskip
\noindent\emph{Step 2.}
Next we apply a similar argument as in Corollary~\ref{ep-es1} to get point-wise estimate of $F_A$.

For any $x\in B^*_{R_0/2}$, we have $B_{r}(x)\subset B_{2r}(0)$ with $|x|=r$. Combining item (2) in Lemma \ref{bochner} with \eqref{es-YMH*}, we have
\[\De |F_A|\geq -\frac{C\ep^\prime_0}{r^2}|F_A|-\frac{C\ep^\prime_0}{r^2}\]
on $B_{r}(x)$. Therefore,  for any fixed $0<p<\frac{6}{3+\al}$, the Nash-Moser estimate implies
\begin{align*}
r^2|F_A|(x)\leq &C_pr^{2-\frac{3}{p}}\(\int_{B_{2r}(x)}|F_A|^pdy\)^{1/p}+Cr^2\ep'_0\\
\leq &C_pr^{2-\frac{3}{p}}\(\int_{B_{2r}(x)}|y|^\al|F_A|^2dy\)^{1/2}\(\int_{B_{2r}(x)}|y|^{-\frac{\al p}{2-p}}dy\)^{\frac{2-p}{p}}+Cr^2\ep^2_0.\\
\leq & C_pI^{1/2}(r)r^{\frac{1-\al}{2}}+Cr^2\ep'_0
\end{align*}
Applying \eqref{eq:integral}, we finally obtain
\begin{align*}
r^2|F_A|(x)\leq C\ep'_0\(\frac{r}{R_0}\)^{2\kappa_0}+Cr^2\ep'_0,
\end{align*}
where $2\kappa_0=\frac{1+1/b-\al}{2}=\frac{1-\al}{2}+\frac{2-C\ep^2_0}{8}$.
\end{proof}

\subsection{Proof of Theorem~\ref{decay-YMH}: energy bound I}
Now we are ready to prove Theorem~\ref{decay-YMH} for YMH fields $(A,u)$ satisfying energy bound I, which we restate as follows.

\begin{thm}\label{t:decay-I}
Let $(A, u)$ be a  smooth YMH field on $B^*_{R_0}$, which satisfies energy bound I:
\begin{equation}\label{inq3}
\sup_{B_{\rho(y)\subset B_{R_0}}}\(\frac{1}{\rho}\int_{B_{\rho}(y)}|\n_A u|^2+|F_A|dx\)\leq \ep^2_0,
\end{equation}	
where $\ep_0$ is given in Lemma \ref{ep-reg}. Then there exists $r_0\leq R_0/2$ such that
\[r^2\sqrt{|F_A|^2(x)+1}\leq C(\ep^2_0+r^2_0)(\frac{r}{r_0})^{1+2\al},\]
and
\[r\sqrt{|\n_A u|^2 +1}\leq C\(\ep_0+r_0\)\(\frac{r}{r_0}\)^{\frac{8\al+1}{5}},\]
for any $r=|x|\leq r_0$, where $\al=\sqrt{\frac{1}{4}-C(\ep^2_0+r^{2/3}_0)}$.
\end{thm}
\begin{proof}
	
The proof is divided into three steps.\

\medskip
\noindent\emph{Step 1: Improved regularity of $F_A$.}\	
	
From \eqref{inq3}, we see that for any $x\in B^*_{2R_0/3}$ and $r=|x|$,
\[\sup_{B_{\rho(y)\subset B_{r/2}}(x)}\(\frac{1}{\rho}\int_{B_{\rho}(y)}|\n_A u|^2+|F_A|dx\)\leq \ep^2_0.\]
Then the standard $\ep$-regularity theorem (Theorem~\ref{ep-reg}) insures the bound \eqref{es-YMH*}. Hence Lemma \ref{decay-F1}  applies and \eqref{eq:lem4.2} holds on $B^*_{R_0/2}$.

In the following, we work in cylindrical coordinates for convenience. Let $t'_0=-\log(R_0/2)$ and $(t=-\log r, \th)\in [t_0', +\infty)\times \mathbb{S}^2$ . Then \eqref{es-YMH*} and \eqref{eq:lem4.2} becomes
\begin{equation}\label{es-cyl1}
e^{-2t}(|\n_A u|^2(t,\th)+|F_A|(t,\th))\leq C\ep^2_0,
\end{equation}
and
\begin{equation}\label{es-cyl2}
e^{-2t}|F_A|\leq C\ep^2_0 e^{-2\ka_0(t-t'_0)},	
\end{equation}
where $\ka_0\geq\frac{1}{16}$.

\medskip
\noindent\emph{Step 2: Decay estimate for $F_A$.}\	

Now inserting \eqref{es-cyl1} into the differential inequality \eqref{eq3} in Lemma \ref{diff-ineq-YMH}, we find that the function $f=(|F_A|^2+1)^{1/4}$ satisfies
\[\p^2_tf-\p_t f+\De_{\mathbb{S}^2}f+C(\ep^2_0+e^{-2t})f\geq 0\]
on $(t'_0,+\infty)\times \mathbb{S}^2$. Hence the product $\bar{f}=e^{-\frac{1}{2}t}f$ satisfies
\[\p^2_t\bar{f}+\De_{\mathbb{S}^2}\bar{f}-(\frac{1}{4}-C(\ep^2_0+e^{-2t}))\bar{f}\geq 0.\]
Then we may choose a large $t_0>t'_0$, such that
\begin{equation}\label{eq:alpha}
  \al^2:=\frac{1}{4}-C(\ep^2_0+e^{-2t_0})\ge (\frac{1}{2}-\ka_0)^2.
\end{equation}
Then for any $t>t_0$, we have
\begin{equation}\label{eq:f-bar}
\p^2_t\bar{f}+\De_{\mathbb{S}^2}\bar{f}-\al^2\bar{f}\geq 0,
\end{equation}
and by \eqref{es-cyl2},
\begin{equation}\label{eq:f-bar-1}
 \bar{f}(t,\th)\leq C\ep^2_0e^{\ka_0 t_0'}e^{(\frac{1}{2}-\ka_0) t}+Ce^{-\frac{1}{2}t}.
\end{equation}

Next for any $m\in \mathbb{Z}^+$ and $T>mt_0$, we consider the comparison function
\[h(t)=C(\ep^2_0e^{m\ka_0 t_0}e^{(\frac{1}{2}-\ka_0) T}+e^{-\frac{1}{2}T})e^{-\al(T-t)}+C(\ep^2_{0}e^{\frac{m}{2} t_0}+e^{-\frac{m}{2}t_0})e^{-\al(t-mt_0)}.\]
which satisfies
\[ \p^2_th+\De_{\mathbb{S}^2}h-\al^2h= 0\]
on $[mt_0, T]\times \mathbb{S}^2$. Inequality \eqref{eq:f-bar-1} guarantees that $h$ also satisfies the boundary condition
\[\bar{f}(mt_0,\th)\leq h(mt_0),\quad \bar{f}(T,\th)\leq f(T).\]
Therefore, the comparison principle gives
\[\bar{f}(t,\th)\leq h(t).\]
on $[mt_0,T]\times \mathbb{S}^2$.

In particular, by \eqref{eq:alpha}, we have $\frac{1}{2}-\ka_0-\al<0$. Therefore, we may let $T\to +\infty$, yielding
\begin{equation}\label{eq:FA}
  e^{-2t}\sqrt{|F_A|^2+1}=e^{-t}\bar{f}^2\leq C(\ep^2_{0}+e^{-2mt_0})e^{-(2\al+1)(t-mt_0)}.
\end{equation}
By choosing $r_0=e^{-3t_0}$, estimate \eqref{eq:FA} with $m=3$ gives the desired estimate of $F_A$.

\medskip
\noindent\emph{Step 3: Decay estimate for $\n_A u$.}\	

By \eqref{eq:FA} with $m=1$, we have
\begin{align*}
e^{-2t}|F_A|^2(t,\th)\leq C(\ep^4_0e^{2t_0}+e^{-2t_0})e^{- 4\al (t-t_0)},
\end{align*}
for all $t\ge t_0$. Inserting into inequality \eqref{eq4} in Lemma \ref{diff-ineq-YMH}, we find that on $[t_0,\infty)\times \mathbb{S}^2$ the function $\tilde{g}=e^{-\frac{1}{2}t}g$ satisfies
\begin{equation}\label{eq-cyl3}
\p^2_t \tilde{g}+\De_{\mathbb{S}^2}\tilde{g}-\al^2\tilde{g}\geq -C(\ep^4_0e^{2t_0}+e^{-2t_0}) e^{-4\al (t-t_0)}\tilde{g},	
\end{equation}
and
\begin{align*}
	\tilde{g}(t,\th)=e^{-\frac{1}{2}t}(|\n_A u|^2+1)^{\frac{1}{2}(\frac{1}{2}+\de)}\leq C\ep_0 e^{\de t}+Ce^{-\frac{1}{2}t}.
\end{align*}
In particular, we may assume that $\al> \frac{1}{4}$ by choosing $\ep_0$ and $r_0$ small enough. Then for $t\geq 3t_0$, the term on the right hand side of the inequality \eqref{eq-cyl3} satisfies
\[Q:=-C(\ep^4_0e^{2t_0}+e^{-2t_0}) e^{-4\al t}\tilde{g}\geq -C(\ep_0e^{\de t}+e^{-\frac{1}{2}t}).\]

Now for $T\geq 3t_0$, we consider the comparison function
\[\tilde{h}(t)=4C(\ep_0e^{\de T}+e^{-\frac{1}{2}T})e^{-\al(T-3t_0)}+4C(\ep_{0}e^{3\de t_0}+e^{-3/2t_0})e^{-\al(t-3t_0)}-2Q.\]
 for $T>3t_0$. A simple computation shows that $\tilde{\ga}=\tilde{g}-\tilde{h}$ solves
\begin{equation*}
\begin{cases}
\p^2_t\tilde{\ga}+\De_{g_{\mathbb{S}^2}}\tilde{\ga}-\al^2\tilde{\ga}\geq 0, \\[1ex]
\tilde{\ga}(t_0,\th)\leq 0, \\[1ex]
\tilde{\ga}(T,\th)\leq 0
\end{cases}	
\end{equation*}
on $[3t_0,T]\times \mathbb{S}^2$. Therefore, the comparison principle gives
\[\tilde{g}\leq \tilde{h}\]
on $[3t_0,T]\times \mathbb{S}^2$. Then by taking $\de=\frac{1}{8}<\al$ and letting $T\to +\infty$, we get
\[e^{-(1/2+\de)t}g(t,\th)\leq C(\ep_0 +e^{-(1/2+\de)3t_0})e^{-(\al+\de)(t-3t_{0})}.\]
Or equivalently,
\[e^{-t}\sqrt{|\n_A u|^2+1}\leq C(\ep_0+e^{-3t_0})e^{-2\frac{\al+\de}{1+2\de}(t-3t_0)}.\]
This gives the desired estimate of $\n_A u$ and the proof is completed.
\end{proof}

\medskip
\section{Decay estimate of 3d YMH fields: energy bound II}\label{s:ebound-II}

In this section, we prove the main Theorem~\ref{decay-YMH} for 3d YMH fields satisfying energy bound II \eqref{Mor-c1}.

\subsection{Monotonicity formula}

Let $(A,u)$ be a YMH field on $B_1\subset \Real^3$, which is equipped with the Euclidean metric $g_{euc}$. We call the following tensor as the twisted stress energy tensor of $u$
\[S_A(u)=\<\n_A u\otimes \n_A u\>-\frac{1}{2}|\n_A u|^2g_{euc}.\]
In fact, if $A$ is a flat connection and the Higgs potential vanishes, $S_A(u)$ is just the usual stress energy tensor for harmonic maps.

The divergence of $S:=S_A(u)$ satisfies
\begin{align*}
(\mbox{div}S)_j=&\<\n_{A_i}\n_{A_i} u, \n_{A_j}u\>+\<\n_{A_i} u, \n_{A_i}\n_{A_j} u\>-\<\n_{A_i}u, \n_{A_j}\n_{A_i} u\>\\
=&-\<\n_A^*\n_A u,\n_j u\>+\<(F_A)_{ij}\cdot u, \n_{A_i} u\>\\
=&\<\n V(u),\n_{A_j} u\>+\<(F_A)_{ij}, u^*(\n_{A_i} u)\>,
\end{align*}
where we have used the formula
\[\n_{A_i}\n_{A_j} u-\n_{A_j}\n_{A_i} u=(F_A)_{ij}\cdot u.\]
Thus, for any $0<R\leq 1$ and vector field $X\in C^\infty(B_R, \Real^3)$, we have
\begin{equation}\label{sta-YMH}
\begin{aligned}
&\int_{B_R}\<S, \n X\>dx-\int_{\p B_R}S(X,\frac{\p}{\p r})dV_{\p B_R}\\
=&-\int_{B_R}\<\n V(u), \n_A u(X)\>+\<F_A(X), u^*(\n_A u)\>dx.
\end{aligned}
\end{equation}
Here $dV_{\p B_R}$ denotes the volume form of $\p B_R$.

The following lemma says that, for a smooth YMH field on the punctured ball $B^*_1$ with bounded YMH energy, the formula \eqref{sta-YMH} still holds true for a vector field $X$ with $X(x)=O(|x|)$.
\begin{lem}\label{pre-lem}
Let $(A,u)$ be a smooth YMH field on $B^*_1$ with bounded YMH energy. Then for any vector field $X\in C^\infty(B_1, \Real^3)$ with $X(x)=O(|x|)$ when $x\to 0$,
then the formula \eqref{sta-YMH} still holds true in the whole $B_R$ for any $0<R\leq 1$.
\end{lem}
\begin{proof}
For any $\ep>0$, let $\eta_\ep$ be a cut-off function such that $\eta_\ep(x)=0 $ in $B_\ep(0)$, $\eta_\ep(x)=1$ in $B_1\setminus B_{2\ep}$ and $|\n \eta_\ep(x)|\leq \frac{C}{\ep}$. Then for any $X\in C^\infty(B_1, \Real^3)$, setting $Y=\eta_\ep(x) X(x)$ and applying \eqref{pre-lem}, we get
\begin{equation}\label{eq:lem5.1}
  \begin{aligned}
&\int_{B_R}\<S, \n Y\>dx-\int_{\p B_R}S(Y,\frac{\p}{\p r})dV_{\p B_R}\\
=&-\int_{B_R}\<\n V(u), \n_A u(Y)\>+\<F_A(Y), u^*(\n_A u)\>dx.
\end{aligned}
\end{equation}

By assumption $X(x)=O(|x|)$, it follows
\begin{align*}
& |\int_{B_R}\<S, \n Y\>dx-\int_{B_R}\eta_\ep\<S, \n X\>dx|= |\int_{B_{2\ep}\setminus B_\ep}\<S, \n \eta_\ep\otimes X\>dx|\\
 \leq& C\int_{B_{2\ep}\setminus B_\ep}|S|dx \leq C\int_{B_{2\ep}\setminus B_\ep}|\n_A u|^2dx.
\end{align*}
Since $(A,u)$ has bounded YMH energy on the whole ball $B_1$, the desired results follows by letting $\ep\to 0$ in \eqref{eq:lem5.1}.
\end{proof}

As an application, we get a monotonicity formula for 3d YMH fields.

\begin{lem}\label{mono-formu}
Let $(A,u)$ be a smooth YMH field on $B^*_1\subset \Real^3$ with bounded YMH energy. Then for any $B_R(y)\subset B_1$ and $\de\in (0,1)$, we have
\begin{equation}\label{mono}
\frac{1}{\si}\int_{B_\si(y)}|\n_A u|^2dx\leq e^{\frac{1}{1-\de}R^{1-\de}}\frac{1}{R}\int_{B_R(y)}|\n_A u|^2dx+2\de^{-1}e^{\frac{1}{1-\de}}\Lambda(R)R^{\de},
\end{equation}
where $\si<R$ and $\Lambda(R)=\int_{B_R(y)}(|F_A|^2+|\n V(u)|^2)dx$.
\end{lem}
\begin{proof}
Let $X=r\frac{\p}{\p r}$ and hence $\n X=g_{euc}$. Then $X=O(|x|)$ and Lemma \ref{pre-lem} applies, yielding
\begin{equation}\label{eq:lem5.2-1}
  \begin{aligned}
&\int_{B_\si}\<S,g_{euc}\>dx-\int_{\p B_\si}r\<S, \frac{\p}{\p r}\otimes\frac{\p}{\p r}\>dV_{\p B_\si}\\
=&-\int_{B_\si} r\<\n V(u), \n_Au(\frac{\p}{\p r})\>dx -\int_{B_\si}r\<u^*(\n_A u), F_A(\frac{\p}{\p r})\>dx.
\end{aligned}
\end{equation}
By definition of $S$, the left hand side of \eqref{eq:lem5.2-1} can be written as
\begin{equation}\label{eq:lem5.2-2}
  \begin{aligned}
LHS=&-\frac{1}{2}\int_{B_\si}|\n_A u|^2dx+\frac{1}{2}\int_{\p B_\si}r|\n_A u|^2dV_{\p B_\si}\\
&-\int_{\p B_\si}r|\n_A u(\frac{\p}{\p r})|^2dV_{\p B_\si}.
\end{aligned}
\end{equation}
The right hand side of \eqref{eq:lem5.2-1} satisfies
\begin{equation}\label{eq:lem5.2-3}
  \begin{aligned}
|RHS| \leq & \si\int_{B_\si}(|F_A|+|\n V(u)|)|\n _A u|dx\\
\leq &\si^{1-\de}\int_{B_\si}|\n_A u|^2dx+\si^{1+\de}\int_{B_\si}\(|F_A|^2+|\n V(u)|^2\)dx
\end{aligned}
\end{equation}
for any $\de\in (0,1)$.

Denote $E(\si)=\int_{B_\si}|\n_A u|^2dx$ and $\Lambda(\si)=\int_{B_\si}|F_A|^2+|\n V(u)|^2dx$. Then we get from \eqref{eq:lem5.2-1}, \eqref{eq:lem5.2-2} and \eqref{eq:lem5.2-3} that
\begin{equation}
  \begin{aligned}
   & \si E^\prime(\si)-E(\si)-2\si\int_{\p B_\si}|\n_A u(\frac{\p}{\p r})|^2dV_{\p B_\si}\\
 \ge& -2\si^{1-\de}E(\si)- 2\si^{1+\de}\Ld(\si).
  \end{aligned}
\end{equation}
Then a simple calculation shows that the function $f(\si)=e^{\frac{1}{1-\de}\si^{1-\de}}\frac{1}{\si}E(\si)$ satisfies
\begin{align*}
f^\prime(\si)\geq \frac{2}{\si}\int_{\p B_{\si}}|\n_Au(\frac{\p}{\p r})|^2dV_{\p B_\si}-2e^{\frac{1}{1-\de}}\Lambda(\si) \si^{\de-1}.
\end{align*}
Integrating from $\si$ to $R$, we get
\[f(\si)+2\int_{B_R}r^{-1}|\n_Au(\frac{\p}{\p r})|^2dx\leq f(R)+2\de^{-1}e^{\frac{1}{1-\de}}\Lambda(R)(R^\de-\si^\de),\]
from which \eqref{mono} follows.
\end{proof}

By using the monotonicity formula \eqref{mono}, we obtain an improved $\ep$-regularity theorem for 3d YMH fields with energy bound II.
\begin{thm}\label{ep-3D}
There exists an $\ep_0>0$ such that for any smooth YMH field $(A,u)$ on $B^*_1$, satisfying
\[\frac{1}{R_0}\int_{B_{R_0}}|\n_A u|^2dx+R_0^{\ka}\(\int_{B_{R_0}}|F_A|^2+|\n V(u)|^2dx\)^{\frac{1}{2}}\leq \ep_0^2,\]
where $R_0\leq 1$ and $\ka\in (0,\frac{1}{2})$, for any $x\in B^*_{R_0/2}$  then we have
\begin{equation}\label{impr-es-YMH}
r^2|\n_A u|^2(x)+r^{\frac{3}{2}+\ka}|F_A|(x)\leq C_\ka\ep^2_0
\end{equation}
for a constant $C_\ka$ depending on the geometry of $N$ and $\ka$.
\end{thm}
\begin{proof}
For any $y\in B_{R_0/2}$ and $\rho<R_0/2$, we can derive from Lemma \ref{mono-formu} with $\de=2\ka$ that
\begin{align*}
&\frac{1}{\rho}\int_{B_{\rho}(y)}|\n_Au|^2dx\\
\leq& C_\ka\frac{2}{R_0}\int_{B_{R_0/2}(y)}|\n_A u|^2dx+C_\ka R_0^{2\ka}\int_{B_{R_0/2}(y)}|F_A|^2+|\n V(u)|^2dx\\
\leq &C_\ka\frac{1}{R_0}\int_{B_{R_0}}|\n_A u|^2dx+C_\ka R_0^{2\ka}\int_{B_{R_0}}|F_A|^2+|\n V(u)|^2dx\\
\leq &C_\ka\ep^2_0.
\end{align*}
This implies that for any $x\in B^*_{R_0/2}$,
\[\sup_{B_\rho(y)\subset B_{r}(x)}\(\frac{1}{\rho}|\n_A u|^2dx\)+r^{\ka}\norm{F_A}_{L^2(B_r(x))}\leq C_\ka\ep^2_0.\]
Then we can apply Corollary \ref{ep-es1} with $p=2$ on $B_{r}(x)$ to get the desired estimate \eqref{impr-es-YMH}.
\end{proof}

\subsection{Proof of Theorem~\ref{decay-YMH}: energy bound II}
Now we restate and prove Theorem~\ref{decay-YMH} for YMH fields satisfying energy bound II. The main difference with energy bound I is that we apply Theorem \ref{ep-3D} instead of Lemma~\ref{decay-F1} to get the desired decay estimate of $F_A$.

\begin{thm}\label{t:decay-II}
Let $(A, u)$ be a smooth YMH field on $B^*_{2R_0}$, which satisfies energy bound II:
\begin{equation}\label{en-es-YMH}
\frac{1}{R_0}\int_{B_{R_0}}|\n_A u|^2dx+R_0^{\ka}\(\int_{B_{2R_0}}|F_A|^2+|\n V(u)|^2dx\)^\frac{1}{2}\leq \ep^2_0,
\end{equation}
where $\ka\in (0, \frac{1}{2})$ and $\ep_0$ is the constant from Theorem \ref{ep-3D}. Then there exists a $r_0\leq R_0/2$ such that
\begin{equation*}
  r^2\sqrt{|F_A|^2+1}\leq C_{\ka}(\ep^2_0+r^2_0)(\frac{r}{r_0})^{1+2\al},
\end{equation*}
and
\begin{equation*}
  r\sqrt{|\n_A u|^2 +1}\leq C_\ka\(\ep_0+r_0\)\(\frac{r}{r_0}\)^{\frac{8\al+1}{5}},
\end{equation*}
for any $0<r\leq r_0$, where $\al=\sqrt{\frac{1}{4}-C_\ka(\ep^2_0+r^{2/3}_0)}$. Here the constant $C_\ka$ depends only on the geometry of $N$ and $\ka$.
\end{thm}
\begin{proof}
Since the YMH field $(A,u)$ satisfies the energy bound \eqref{en-es-YMH}, then Theorem \ref{ep-3D} implies
\begin{equation}\label{eq:5.4-1}
  e^{-2t}|\n_A u|^2(t,\th)+e^{-(\frac{3}{2}+\ka)t}|F_A|(t,\th)\leq C_\ka\ep^2_0
\end{equation}
with respect to cylinder coordinates on $B^*_{\frac{3}{2}R_0}$.

Then we follow the same arguments as in the proof of Theorem~\ref{t:decay-I}. The only difference is that in step 1, we use \eqref{eq:5.4-1} instead of \eqref{es-cyl1} and \eqref{es-cyl2}, where now we choose $\ka_0=\frac{1}{2}(\frac{1}{2}-\ka)>0$. The rest arguments in step 2 and 3 are the same and we omit the details.

\end{proof}

\subsection{Decay estimates for pure YM fields and harmonic maps}\label{s:decay-pure}

As a byproduct, we also obtain sharp decay estimates for pure YM fields, which are special cases of YMH fields.
\begin{thm}\label{c:pure-YM}
There exists an $\ep_1$ such that if $A$ is a pure YM connection on $B^*_{R_0}$, which satisfies the energy bound
\[\sup_{B_{\rho(y)\subset B_{R_0}}}\(\frac{1}{\rho}\int_{B_{\rho}(y)}|F_A|dx\)\leq \ep^2_1,\]
then
\[r^2|F_A|(x)\leq C\ep^2_1 \(\frac{r}{R_0}\)^{2}\]
for any $r=|x|<R_0/2$.

\end{thm}

\begin{proof}

The proof is divided into two steps.\

\medskip
\noindent\emph{Step 1: Decay estimate for $F_A$.}\

Since a pure YM connection $A$ can be viewed as a YMH field $(A,u)$ with $V(u)=0$ and $\n_A u=0$, the standard $\ep$-regularity theorem (Theorem~\ref{ep-reg}) gives
\begin{equation}\label{c-decay-YM}
r^2|F_A|\leq C\ep^2_1
\end{equation}
for any $x \in B_{3R_0/4}$. Hence, we can apply Lemma \ref{decay-F1} to get
\begin{equation}\label{c-decay-YM-1}
r^2|F_A|\leq C\ep^2_1\(\frac{r}{R_0}\)^{2\kappa_0}	
\end{equation}
for any $x\in B_{2R_0/3}$ and some $\kappa_0 \in [\frac{1}{16}, 1]$.

On the other hand, since $A$ is a pure YM connection, $F_A$ is now a harmonic 2-form with respect to $D_A$. By Lemma~\ref{Kato-ineq-2} (1), $F_A$ satisfies the Kato inequality
\begin{equation}\label{eq:6-1}
    |\n_A F_A|^2\geq \frac{3}{2}|\n|F_A||^2.
\end{equation}
A simply calculation combining the standard Bochner formula and \eqref{eq:6-1} shows that
\[\De |F_A|^{\frac{1}{2}}\geq -C|F_A||F_A|^\frac{1}{2}.\]
It follows that in cylindrical coordinates, the function $g=e^{-\frac{1}{2}t}|F_A|^{\frac{1}{2}}$ satisfies
\begin{equation}\label{c-eq-YM}
\p^2_t g+\De_{\mathbb{S}^2}g\geq (\frac{1}{4}-Ce^{-2t}|F_A|)g.
\end{equation}
Now inserting \eqref{c-decay-YM} into inequality \eqref{c-eq-YM}, we have
\begin{equation}\label{c-eq-YM-1}
\p^2_t g+\De_{\mathbb{S}^2}g\geq \al^2g,
\end{equation}
where $\al^2=\frac{1}{4}-C\ep^2_1$. Moreover, by estimate \eqref{c-decay-YM-1}, we have $g(t,\th)\leq C\ep_1e^{(\frac{1}{2}-\kappa_0)t}$.

Hence by applying a similar argument as that in step 2 in the proof of Theorem~\ref{t:decay-I}, we get
\begin{align}
g(t,\th)\leq C\ep_1e^{\frac{1}{2}t_0}e^{-\al(t-t_0)}\label{c-decay-FA}
\end{align}
on $[t_0=-\log(R_0/2),+\infty)\times \mathbb{S}^2$.

\medskip
\noindent\emph{Step 2: Sharp decay estimate for $F_A$.}\

The estimate \eqref{c-decay-FA} in step 1 gives
\[e^{-2t}|F_A|\leq C\ep_1e^{-(\al+1/2)(t-t_0)}.\]
Then inserting this decay estimate into inequality \eqref{c-eq-YM}, we find that $g$ satisfies
\[\p^2_t g+\De_{\mathbb{S}^2}g-\frac{1}{4}g\geq -C\ep^3_1e^{\frac{1}{2}t_0}e^{-\si (t-t_0)},\]
where $\si=1+3\al$.

Now, we consider the following comparison function
\begin{align*}
h(t)=&2C\ep_1e^{\frac{1}{2}t_0}\(e^{-\frac{1}{2}(t-t_0)}+e^{-\frac{1}{2}(T-t)}\)-C\ep^3_1e^{\frac{1}{2}t_0}e^{-\si (t-t_0)},
\end{align*}
which satisfies
\[\p^2_t h+\De_{\mathbb{S}^2}h-\frac{1}{4}h\leq -(\si^2-\frac{1}{4})C\ep^3_1e^{\frac{1}{2}t_0}e^{-\si (t-t_0)}\leq C\ep^3_1e^{\frac{1}{2}t_0}e^{-\si (t-t_0)}.\]
The estimate \eqref{c-decay-YM-1} ensure that $g$ satisfies the boundary condition:
\[g(t_0)\leq h(t_0)\quad g(T)\leq h(T).\]

Therefore, by the comparison principle for ODE, we have
\[g(t)\leq 2C\ep_1e^{\frac{1}{2}t_0}\(e^{-\frac{1}{2}(t-t_0)}+e^{-\frac{1}{2}(T-t)}\).\]
Letting $T\to +\infty$, we get
\[|F_A|(t,\th)\leq C\ep^{2}_1e^{2 t_0},\]
and the proof is finished.
\end{proof}

Another byproduct is the sharp decay estimates for 3d harmonic maps.
\begin{thm}\label{c:pure-HM}
There exists an $\ep_2$ such that if $u$ is a smooth harmonic map on $B^*_{R_0}$, which satisfies the energy bound
\[\frac{1}{R_0}\int_{B_{R_0}}|\n u|^2dx\leq \ep^2_2.\]
Then $\n u$ satisfies
\[r|\n u|\leq C\ep_2\frac{r}{R_0}\]
for any $r=|x|\leq R_0/2$.
\end{thm}	
\begin{proof}
The proof is divided into two steps.\

\medskip
\noindent\emph{Step 1: Decay estimate for $\n u$.}\

Since a harmonic map $u$ can be interpreted as a YMH field $(A,u)$ with $V(u)=0$ and $A=0$, the $\ep$-regularity theorem \ref{ep-3D} gives
\begin{equation}\label{es-pure-HM}
r|\n u|(x)\leq C\ep_2
\end{equation}
for any $x\in B^*_{R_0/2}$.

On the other hand, since $u$ is a harmonic map, $du$ is a harmonic 1-form with respect to the induced extrinsic derivative $D$, i.e.
\[Ddu=0\quad \text{and} \quad D^*du=0.\]
So, by the Kato inequality in Lemma \ref{Kato-ineq-1}, $du$ satisfies
\begin{equation}\label{eq:6-2}
    |\n d u|^2\geq \frac{3}{2}|\n|\n u||^2.
\end{equation}
Combining the standard Bochner formula and \eqref{eq:6-2}, we get
\begin{equation*}\label{c-eq-du}
\De |\n u|^{\frac{1}{2}}\geq -C|\n u|^2|\n u|^\frac{1}{2}.
\end{equation*}
It follows that in cylindrical coordinates, the function $\tilde{g}=e^{-\frac{1}{2}t}|\n u|^{\frac{1}{2}}$ satisfies
\begin{equation}\label{c-eq-du-1}
\p^2_t \tilde{g}+\De_{\mathbb{S}^2}\tilde{g}\geq(\frac{1}{4}-e^{-2t}|\n u|^2)\tilde{g}\geq (\frac{1}{4}-C\ep^2_2)\tilde{g}=\al^2\tilde{g},
\end{equation}
where $\al^2=\frac{1}{4}-C\ep^2_2$. By \eqref{es-pure-HM}, we also have $\tilde{g}(t,\th)\leq C\ep^\frac{1}{2}_1$ for $t\geq t_0=-\log (R_0/2)$.

Hence by applying a similar argument as that in step 2 in the proof of Theorem~\ref{t:decay-I}, we have
\begin{equation}\label{c-decay-es-du}
\tilde{g}(t,\th)\leq C\ep^\frac{1}{2}_2e^{-\al(t-t_0)}.
\end{equation}

\medskip
\noindent\emph{Step 2: Sharp decay estimate for $\n u$.}\

By \eqref{c-decay-es-du}, we have
\[e^{-2t}|\n u|^2\leq C\ep^2_2e^{4\al(t-t_0)}.\]
Inserting this estimate into inequality \eqref{c-eq-du}, we have
\[\p^2_t\tilde{g}+\De_{\mathbb{S}^2}\tilde{g}-\frac{1}{4}\tilde{g}\geq -C\ep^\frac{5}{2}_2e^{-5\al (t-t_0)}.\]

Now we consider the following comparison function
\[\tilde{h}(t)=2C\ep^{\frac{1}{2}}_2e^{-\frac{1}{2}(t-t_0)}+2C\ep^{\frac{1}{2}}_2e^{-\frac{1}{2}(T-t)}-C\ep^\frac{5}{2}_2e^{-5\al (t-t_0)},\]
which satisfies
\[\p^2_t \tilde{h}+\De_{\mathbb{S}^2}\tilde{h}-\frac{1}{4}\tilde{h}\leq -(25\al^2-\frac{1}{4})C\ep^\frac{5}{2}_2e^{-5\al (t-t_0)}\leq C\ep^\frac{5}{2}_2e^{-5\al (t-t_0)},\]
Moreover, estimate \eqref{c-decay-es-du} ensure that $\tilde{h}$ satisfies the boundary condition:
\[\tilde{g}(t_0)\leq \tilde{h}(t_0)\quad \tilde{g}(T)\leq \tilde{h}(T).\]
Then applying the comparison principle for ODE, we obtain
\[\tilde{g}(t)\leq 2C\ep^{\frac{1}{2}}_2e^{-\frac{1}{2}(t-t_0)}+2\ep^{\frac{1}{2}}_2e^{-\frac{1}{2}(T-t)}.\]
Letting $T\to +\infty$, we get
\[e^{-t}|\n u|(t,\th)\leq C\ep_2e^{-(t-t_0)},\]
and the proof is completed.
\end{proof}

\section{The removability of singularities for YMH fields}\label{s: remov-singu}

In this section, we apply the decay estimates in Theorem~\ref{decay-YMH} to prove the removable singularity theorems for YMH fields, i.e. Theorem \ref{remov-iso-singu} and Corollary \ref{remov-iso-singu1}. 

First we recall the local existence of Coulomb gauge established by P. Smith and K. Uhlenbeck in \cite{SU19}.
\begin{lem}\label{coulomb-gauge}
Let $B_R$ be a ball in $\Real^n$ with $n\geq 3$, $\mathcal{S}\subset B_R$ be a closed set with finite $n-3$ dimensional Hausdorff measure. There is an $\ep>0$ such that for any connection $A$ that is smooth on $B_R\setminus\mathcal{S}$, if for $p>n$ the curvature $F_A$ satisfies
\[\int_{B_R}|F_A|^pdx\leq \ep^pR^{-2p+3},\]
then there exists a gauge transformation $s$ such that $\bar{A}=s^*A\in W^{1,p}_{loc}(B_R)$ and $\bar{A}$ satisfies the following properties.
\begin{itemize}
\item[$(1)$] $\bar{A}$ solves $d^*\bar{A}=0$ in $B_R$,
\item[$(2)$] There exists a constant C depending only on $n$ and $p$ such that
\begin{align*}
	&R^{1-3/p}\norm{\bar{A}}_{L^p(B_{R/2})}+R^{2-3/p}\norm{\n \bar{A}}_{L^p(B_{R/2})}\\
\leq& CR^{2-3/p}\norm{F_{\bar{A}}}_{L^p(B_R)}\leq C\ep.
\end{align*}
\end{itemize}
\end{lem}


\begin{proof}[Proof of Theorem \ref{remov-iso-singu}]


Since $(A,u)$ satisfies energy bounds \eqref{Mor-c0} or \eqref{Mor-c1}, by applying decay estimate \eqref{main-theorem-curvature} for $F_A$ in Theorem~\ref{decay-YMH} with $\ep_0$ and $r_0$ both small enough, we have
\begin{equation}\label{small-FA}
\int_{B_{r_0}}|F_A|^pdx\leq \ep^pr_0^{-2p+3}
\end{equation}
where we choose $p>6$ for convenience and $\ep$ is from Lemma \ref{coulomb-gauge}. Then by Lemma \ref{coulomb-gauge}, there exists a gauge transformation $s$ on $B^*_{r_0}$ such that under this gauge transformation we have
\begin{equation}\label{eq-A}
d^*A_1=0 \quad\text{in}\,\, B_{r_0},
\end{equation}
\begin{equation}\label{es-A}
\norm{A_1}_{W^{1,p}(B_{r_0/2})}\leq C(p,r_0)\norm{F_{A_1}}_{L^p(B_{r_0})},
\end{equation}
where $A_1=s^*A=s^{-1}ds+s^{-1}As$.

Moreover, decay estimate \eqref{main-theorem-section} for $\n_Au$ in Theorem~\ref{decay-YMH} implies that
\begin{equation}\label{es-u1}
\norm{u_1}_{W^{1,p}(B_{r_0/2})}\leq C(r_0),
\end{equation}
where we set $u_1=s^*u=s^{-1}\cdot u$.

On the other hand, using \eqref{eq-A}, we know from equation \eqref{eq-YMH3} that the YMH field $(A_1,u_1)$ satisfies the following extrinsic equation on $B_{r_0/2}$
\begin{equation}\label{eq-YMH-extrinsic}
\begin{cases}
\De A_1+[A_1,dA_1]+[A_1,[A_1,A_1]]=-u^{*}_1(\n_{A_1}u_1),\\
\Delta u_1+2A_1du_1+A_1(A_1u_1)=\Ga(u_1)(\n_{A_1} u_1,\n_{A_1} u_1)+\n V(u_1).
\end{cases}
\end{equation}
Then, by applying estimates \eqref{es-A}-\eqref{es-u1}, we can derive form equation \eqref{eq-YMH-extrinsic} that
\[(\De A_1, \De u_1)\in L^\frac{p}{2}_{loc}(B_{r_0/2}).\]
Hence the $L^p$-estimates for elliptic equations gives
\[(A_1,u_1)\in W^{2,\frac{p}{2}}_{loc}(B_{r_0/2})\hookrightarrow C^{1,1-\frac{6}{p}}_{loc}(B_{r_0/2}).\]

It follows that the singularity at the origin is removable and the fiber extends to the ball $B_1$. Then a standard bootstrap argument implies that $(A_1, u_1)$ is actually smooth.

\end{proof}

Finally, Corollary \ref{remov-iso-singu1} follows directly from H\"older's inequality and Theorem \ref{remov-iso-singu}, and we omit the proof.

\section*{Acknowledgements}
B. Chen is partially supported by NSFC (Grant No. 12301074) and Guangzhou Basic and Applied Basic Research Foundation (Grant No. 2024A04J3637). C. Song is partially supported by NSFC no. 12371061 and Natural Science Foundation of Fujian Province of China No. 2021J06005.


\end{document}